\documentclass[11pt,reqno]{article}

\title{Increasing the smoothness of vector and Hermite subdivision schemes}
\author{%
{\sc
Caroline Moosm{\"u}ller}\thanks{Corresponding author. Chair of Digital Image Processing, University of Passau,
Innstra{\ss}e 43, 94032 Passau, Germany. C.M.\ is now with Johns Hopkins University. \texttt{cmoosmueller@jhu.edu}}\;,
{\sc Nira Dyn}\thanks{School of Mathematical Sciences, Tel-Aviv University,
Tel-Aviv 69978, Israel. \texttt{niradyn@post.tau.ac.il}}
}

\date{}
\usepackage{amsmath,amsfonts,amssymb,amsthm,tensor,graphicx,caption,subcaption,overpic,enumerate,relsize,tikz,url}
\usepackage[round]{natbib}
\usepackage{hyperref}
\usepackage{cleveref}
\usetikzlibrary{matrix,arrows,decorations.pathmorphing}
\usepackage[a4paper]{geometry}

\theoremstyle{plain}
\newtheorem{theorem}{Theorem}
\newtheorem{lemma}[theorem]{Lemma}
\newtheorem{corollary}[theorem]{Corollary}

\newtheorem{alg}[theorem]{Procedure}

\theoremstyle{definition}
\newtheorem{definition}[theorem]{Definition}
\newtheorem{example}[theorem]{Example}

\newtheorem{remark}[theorem]{Remark}

\newcommand{\A}{\mathbf{A}}
\newcommand{\B}{\mathbf{B}}
\newcommand{\C}{\mathbf{C}}
\newcommand{\D}{\Delta}
\newcommand{\ab}{\boldsymbol{\alpha}}
\newcommand{\bb}{\boldsymbol{\beta}}
\newcommand{\bc}{\mathbf{c}}
\newcommand{\m}{\mathbb{Z}}
\newcommand*\Bell{\ensuremath{\boldsymbol\ell}}
\newcommand{\vs}{VSS}
\newcommand{\hs}{HSS}

\makeatletter
\newcommand{\Spvek}[2][r]{%
  \gdef\@VORNE{1}
  \left(\hskip-\arraycolsep%
    \begin{array}{#1}\vekSp@lten{#2}\end{array}%
  \hskip-\arraycolsep\right)}

\def\vekSp@lten#1{\xvekSp@lten#1;vekL@stLine;}
\def\vekL@stLine{vekL@stLine}
\def\xvekSp@lten#1;{\def\temp{#1}%
  \ifx\temp\vekL@stLine
  \else
    \ifnum\@VORNE=1\gdef\@VORNE{0}
    \else\@arraycr\fi%
    #1%
    \expandafter\xvekSp@lten
  \fi}
\makeatother

\begin{document}
\maketitle
\begin{abstract}
 
In this paper we suggest a method for transforming a vector subdivision scheme generating $C^{\ell}$ limits to another such scheme of the same dimension, generating $C^{\ell+1}$ limits.
In scalar subdivision, it is well known that
a scheme generating $C^{\ell}$ limit curves can be transformed to a new scheme producing $C^{\ell+1}$ limit curves by
multiplying the scheme's symbol with the \emph{smoothing factor} $\tfrac{z+1}{2}$.
We extend this approach to vector and Hermite subdivision schemes, by manipulating symbols.
The algorithms presented in this paper allow to construct vector (Hermite) subdivision schemes of arbitrarily high regularity from a convergent vector scheme (from a Hermite scheme whose Taylor scheme is convergent with limit functions of vanishing first component).

\end{abstract}

\section{Introduction}

Subdivision schemes are algorithms which iteratively refine discrete input data and produce smooth curves or surfaces in the limit. The regularity of the limit curve resp.\ surface is a topic of high interest.

In this paper we are concerned with the \emph{stationary} and \emph{univariate} case, i.e.\ with
subdivision schemes using the same set of coefficients (called \emph{mask}) in every refinement step and which have curves as limits. We study two types of such schemes: \emph{vector} and \emph{Hermite} subdivision schemes.

The mostly studied schemes are scalar subdivision schemes with real-valued sequences as masks. These schemes are in fact a special case of vector subdivision, with matrix-valued sequences as masks which refine sequences of vectors. For vector subdivision schemes many results concerning convergence and smoothness are available.
An incomplete list of references is \citet{cavaretta91,charina05,dyn92,dyn91,dyn02,micchelli98,sauer02}.

In Hermite subdivision the refined data
is also a sequence of vectors interpreted as function and derivatives values. This results in level-dependent vector subdivision, where the convergence of a scheme already includes the regularity of the limit curve. Corresponding literature can be found in \citet{dubuc06,dubuc05,dyn95,dyn99,guglielmi11,han05,merrien12}
and references therein. Note that here we consider \emph{inherently stationary} Hermite schemes \citep{conti14}, where the level-dependence arises only from the specific interpretation of the input data. Inherently non-stationary Hermite schemes are discussed e.g.\ in \citet{conti16}.

The convergence and smoothness analysis of subdivision schemes is strongly connected to the existence of the \emph{derived scheme} or in the Hermite case to the \emph{Taylor scheme}. The derived scheme (the Taylor scheme) are obtained by an appropriate factorization of the symbols \citep{dyn02,charina05} (\citet{merrien12}).
In the scalar and vector case we have the following result: If the derived scheme produces $C^{\ell}\enspace (\ell \geq 0)$ limit curves, then the original scheme produces $C^{\ell+1}$ limit curves, see \citet{dyn02,charina05}. In the Hermite case, in addition to the assumption that the Taylor scheme is $C^{\ell}$,
we also need that its limit functions have vanishing first component \citep{merrien12}.
These results are an essential tool in our approach for obtaining schemes with increased smoothness.

We start from a scheme which is known to have a certain regularity as the derived scheme (the Taylor scheme) of a new, to be computed scheme. By the above result, the regularity of the new scheme is increased by $1$.
This idea comes from univariate scalar subdivision, where it is well known that a scheme with symbol $\boldsymbol{\alpha}^{\ast}(z)$ is the derived scheme of $\boldsymbol{\beta}^{\ast}(z)=\tfrac{1+z}{2}z^{-1}\boldsymbol{\alpha}^{\ast}(z)$ \citep{dyn02}, and thus if $S_{\boldsymbol{\alpha}}$ generates $C^{\ell}$ limits,
$S_{\boldsymbol{\beta}}$ generates limits which are $C^{\ell+1}$.

It is possible to generalize this process to obtain vector (Hermite) subdivision schemes of arbitrarily high smoothness from a convergent vector scheme (a Hermite scheme, whose Taylor scheme is convergent with limit functions of vanishing first component).

We would like to mention other approaches which increase the regularity of subdivision schemes: It is known that the de Rham transform \citep{dubuc08} of some Hermite schemes increases the regularity by $1$, see \citet{conti14}. In contrast to our approach,
it is not clear if this procedure can be iterated to obtain schemes of higher regularity. 
Nevertheless, in the examples listed in \citet{conti14}, the de Rham approach increases the support only by $1$, whereas our procedure for increasing the smoothness has the drawback of producing Hermite schemes with large supports (see \Cref{supportlength}, \Cref{example} and \Cref{derham}).
Also, the authors of \citet{dubuc08} use geometric ideas, such as corner cutting. Our approach, on the other hand, is of an algebraic nature as it manipulates symbols.

A recent result which increases the regularity of a Hermite scheme, but not vector schemes, is presented in \citet{merrien16}. This is different from our approach, as it also increases the dimension of the matrices of the mask and the dimension of the refined data.

We would also like to mention the paper \citet{sauer03}, which gives a detailed discussion of how to generalize the procedure for increasing the smoothness of a scalar subdivision scheme from the univariate to the multivariate scalar setting. Although vector subdivision schemes appear naturally
in the analysis of smoothness of multivariate scalar schemes, yet the aim in \citet{sauer03} is to increase the smoothness of scalar schemes.

Our paper is organized as follows. In \Cref{background} we introduce the notation used throughout this text and recall some definitions concerning subdivision schemes. \Cref{sec:scalar} presents the well known procedure for increasing the smoothness of univariate scalar subdivision schemes \citep{dyn02}.
However, we introduce new notation, to emphasize the analogy to the procedures we presented in
\Cref{sec:vector,sec:hermite} for vector and Hermite schemes, respectively. We conclude by two examples, applying our procedure to an interpolatory Hermite scheme of \citet{merrien92} and to a Hermite scheme of de Rham-type \citep{dubuc08}, and obtain schemes with limit curves of regularity $C^2$ and $C^3$, respectively.

\section{Notation and background}
\label{background}

In this section we introduce the notation which is used throughout this paper and recall some known facts about scalar, vector and Hermite subdivision schemes.

Vectors in $\mathbb{R}^p$ will be labeled by lowercase letters $c$. The standard basis is denoted by $e_1,\ldots,e_p$.
Sequences of elements in $\mathbb{R}^p$ are denoted by boldface letters $\bc=\{c_i \in \mathbb{R}^p: i \in \m \}$. The space of all such sequences is $\ell(\mathbb{R}^p)$.

We define a \emph{subdivision operator} $S_{\ab}: \ell(\mathbb{R}^p) \to \ell(\mathbb{R}^p)$ with a scalar mask $\ab \in \ell(\mathbb{R})$ by
\begin{equation}\label{scalaroperator}
(S_{\ab}\bc)_i=\sum_{j \in \m} \alpha_{i-2j}c_j, \quad i \in \m, \: \bc \in \ell(\mathbb{R}^p).
\end{equation}
We study the case of finitely supported masks, with support contained in $[-N,N]$. In this case the sum in \cref{scalaroperator} is finite and the scheme is local.

We also consider matrix-valued masks. To distinguish them from the scalar case, we denote
matrices in $\mathbb{R}^{p \times p}$ by uppercase letters. Sequences of matrices are denoted by boldface letters $\A=\{A_i \in \mathbb{R}^{p \times p}: i \in \m \}$. 

We define a \emph{vector subdivision operator} $S_{\A}: \ell(\mathbb{R}^p) \to \ell(\mathbb{R}^p)$ with a finitely supported matrix mask $\A \in \ell(\mathbb{R}^{p \times p})$ by
\begin{equation}\label{matrixoperator}
(S_{\A}\bc)_i=\sum_{j \in \m} A_{i-2j}c_j, \quad i \in \m, \: \bc \in \ell(\mathbb{R}^p).
\end{equation} 
We define three kinds of subdivision schemes:
\begin{definition}

\
\begin{enumerate}
\item A \emph{scalar subdivision scheme} is the procedure of constructing $\bc^n\: (n\geq 1)$ from input data $\bc^0 \in \ell(\mathbb{R}^p)$ by the rule $\bc^n=S_{\ab}\bc^{n-1}$, where $\ab \in \ell(\mathbb{R})$ is a scalar mask.
\item A \emph{vector subdivision scheme} (\vs) is the procedure of constructing $\bc^n\: (n\geq 1)$ from input data $\bc^0 \in \ell(\mathbb{R}^p)$ by the rule $\bc^n=S_{\A}\bc^{n-1}$, where $\A$ is a matrix-valued mask.
\item A \emph{Hermite subdivision scheme} (\hs) is the procedure of constructing $\bc^n\: (n\geq 1)$ from $\bc^0\in \ell(\mathbb{R}^p)$ by the rule $D^n\bc^n=S_{\A}D^{n-1}\bc^{n-1}$, where $\A$ is a matrix-valued mask and $D$ is the dilation matrix
\begin{equation*}
D=\begin{pmatrix} 1 & & & \\  & \frac{1}{2} & & \\ & & \ddots & \\& & & \frac{1}{2^{p-1}}  \end{pmatrix}.
\end{equation*}
\end{enumerate}
\end{definition}
The difference between scalar and vector subdivision lies in the dimension of the mask. In scalar subdivision the components of $\bc$ are refined \emph{independently} of each other. This is not the case in vector subdivision. Note also that scalar schemes are a special case of vector schemes with mask $A_i=\alpha_iI_{p}$, where $I_p$ is the $(p \times p)$ unit matrix.
In Hermite subdivision, on the other hand, the components of $\bc$ are interpreted as function and derivatives values up to order $p-1$. This is represented by the matrix $D$. In particular, Hermite subdivision is a level-dependent case of vector subdivision:
$\bc^n=S_{\hat{\A}_n}\bc^{n-1}$ with $\hat{\A}_n=\{D^{-n}A_iD^{n-1}:i \in\m\}$.

On the space $\ell(\mathbb{R}^p)$ we define a norm by
\begin{equation*}
\|\bc\|_{\infty}=\sup_{i \in \m}\|c_i\|,
\end{equation*}
where $\|\cdot\|$ is a norm on $\mathbb{R}^p$. The Banach space of all bounded sequences is denoted by $\ell^{\infty}(\mathbb{R}^p)$.
A subdivision operater $S_{\ab}$ with finitely supported mask, restricted to a map $\ell^{\infty}(\mathbb{R}^p) \to \ell^{\infty}(\mathbb{R}^p)$ has an induced operator norm:
\begin{equation*}
\|S_{\ab}\|_{\infty}=\sup\{\|S_{\ab}\bc\|_{\infty}: \bc \in \ell^{\infty}(\mathbb{R}^p) \text{ and } \|\bc\|_{\infty}=1\}.
\end{equation*}
This is also true for subdivision operators with matrix masks.

Next we define convergence of scalar, vector and Hermite subdivision schemes. We start with scalar and vector schemes:
\begin{definition}
A scalar (vector) subdivision scheme associated with the mask $\ab$ ($\A$) is \emph{convergent in $\ell^{\infty}(\mathbb{R}^p)$}, also called $C^0$, if for all input data $\bc^0 \in \ell^{\infty}(\mathbb{R}^p)$ there exists a function $\Psi \in C(\mathbb{R},\mathbb{R}^p)$, such that the sequences $\bc^n=S_{\ab}^n\bc^0$ ($\bc^n=S_{\A}^n\bc^0$) satisfy
\begin{equation*}
\sup_{i\in \m}\|c^n_i-\Psi(\tfrac{i}{2^n})\| \to 0, \quad \text{as } n \to \infty,
\end{equation*}
and $\Psi\neq 0$ for some $\bc^0 \in \ell^{\infty}(\mathbb{R}^p)$.
We say that the scheme is $C^{\ell}$, if in addition $\Psi$ is $\ell$-times continuously differentiable for any initial data.
\end{definition}
In \Cref{sec:hermite} we consider {\hs}s which refine function and first derivative values. 
The case of point-tangent data is treated componentwise. With this approach it is sufficient to consider convergence for data in $\ell(\mathbb{R}^2)$.

In order to distinguish between the convergence of {\vs}s and the convergence of {\hs}s, we use the notation introduced in \citet{conti14}:
\begin{definition}\label{hermiteconv}
A {\hs} associated with the mask $\A$ is said to be \emph{$HC^{\ell}$ convergent} with $\ell\geq 1$, if for any input data $\bc^0 \in \ell^{\infty}(\mathbb{R}^2)$, there exists a function $\Psi={\Psi^0 \choose \Psi^1}$ with $\Psi^0 \in C^{\ell}(\mathbb{R},\mathbb{R})$
and $\Psi^1$ being the derivative of $\Psi^0,$ such that the sequences $\bc^n=D^{-n}S_{\A}^n\bc^{0}$, $n \geq 1$, satisfy
\begin{equation*}
\sup_{i \in \m}\|c^n_i-\Psi(\tfrac{i}{2^n})\| \to 0, \quad \text{as } n\to \infty.
\end{equation*}
\end{definition}
Note that in contrast to the vector case, a {\hs} is convergent only if the limit already possesses a certain degree of smoothness.

We conclude by recalling some facts about the generating function of a sequence $\bc$, which is the formal Laurent series
\begin{equation*}
\bc^{\ast}(z)=\sum_{i \in \m}c_iz^i.
\end{equation*}
The generating function of a mask of a subdivision scheme is called the symbol of the scheme.
It is easy to see (e.g.\ in \citet{dyn02}) that $\bc^{\ast}(z)$ has the following properties:
\begin{lemma}\label{lem:maskproperties}
Let $\bc$ be a sequence and let $\ab$ be a scalar or a matrix mask. By $\D$ we denote the forward-difference operator $(\D\bc)_i=c_{i+1}-c_{i}$. Then we have:
\begin{equation*}
(\D\bc)^{\ast}(z)=(z^{-1}-1)\bc^{\ast}(z) \text{ and } (S_{\ab}\bc)^{\ast}(z)=\ab^{\ast}(z)\bc^{\ast}(z^2).
\end{equation*}
Furthermore, for finite sequences we have the equalities
\begin{align*}
&\bc^{\ast}(1)=\sum_{i \in \m}c_{2i}+\sum_{i \in \m}c_{2i+1} \quad \text{and} \quad
\bc^{\ast}(-1)=\sum_{i \in \m}c_{2i}-\sum_{i \in \m}c_{2i+1},\\
&{\bc^{\ast}}'(1)=\sum_{i \in \m}c_{2i}(2i)+\sum_{i \in \m}c_{2i+1}(2i+1) \quad \text{and} \quad
{\bc^{\ast}}'(-1)=\sum_{i \in \m}c_{2i+1}(2i+1)-\sum_{i \in \m}c_{2i}(2i).
\end{align*}
\end{lemma}

\section{Increasing the smoothness of scalar subdivision schemes}
\label{sec:scalar}

In this section we recall a procedure increasing the smoothness of scalar subdivision schemes, which is realized by the smoothing factor $\tfrac{z+1}{2}$. The results of this section are taken from Section 4 in \citet{dyn02}.
We introduce notation in order to illustrate the analogy to the procedures we present in \Cref{sec:vector} for {\vs}s.

The condition $\sum_{i \in \m}\alpha_{2i}=\sum_{i \in \m}\alpha_{2i+1}=1$ on the mask $\ab$ is necessary for the convergence of $S_{\ab}$.
In this case $\ab^{\ast}(-1)=0$, implying that $\ab^{\ast}(z)$ has a factor $(z+1)$ and there exists a mask $\partial \ab$ such that
\begin{equation}\label{scalarpartial}
\D S_{\ab}=\tfrac{1}{2}S_{\partial \ab}\D.
\end{equation}
The scalar scheme associated with $\partial \ab$ is called the \emph{derived scheme}.
It is easy to see that 
\begin{equation}\label{defderiv}
(\partial \ab)^{\ast}(z)=2z\tfrac{\ab^{\ast}(z)}{z+1}
\end{equation}
and that $(\partial \ab)^{\ast}$ is a Laurent polynomial.
The convergence and smoothness analysis of a scalar subdivision scheme associated with $\ab$ depends on the properties of $\partial \ab$:
\begin{theorem}\label{scalarsmooth} 
Let $\ab$ be a mask which satisfies $\ab^{\ast}(1)=2$ and $\ab^{\ast}(-1)=0$.
\begin{enumerate}
\item The scalar scheme associated with $\ab$ is convergent if and only if the scalar scheme associated with $\tfrac{1}{2}\partial \ab$ is contractive, namely $\|(\frac{1}{2}S_{\partial \ab})^L\|_{\infty}<1$ for some $L \in \mathbb{N}$.
\item If the scalar scheme associated with $\partial \ab$ is $C^{\ell}\enspace (\ell\geq 0)$ then the scalar subdivision scheme associated with $\ab$ is $C^{\ell+1}$.
\end{enumerate}
\end{theorem}
\Cref{scalarsmooth} allows us to define a procedure for increasing the smoothness of a scalar subdivision scheme:
For a mask $\ab$, define a new mask $\mathcal{I}\ab$ by $(\mathcal{I}\ab)^{\ast}(z)=\tfrac{(1+z)}{2}z^{-1}\ab^{\ast}(z)$. Then $(\mathcal{I}\ab)^{\ast}(-1)=0$ and from \cref{defderiv} we get $\partial(\mathcal{I}\ab)=\ab$ (Note that if $\partial \ab$ exists, then also $\mathcal{I}(\partial \ab)=\ab$).
\begin{corollary}
Let $\ab$ be a mask associated with a $C^{\ell}$ ($\ell \geq 0$) scalar subdivision scheme. Then the mask $\mathcal{I}\ab$ gives rise to a $C^{\ell+1}$ scheme.
\end{corollary}
Therefore, by a repeated application of $\mathcal{I}$, a scalar subdivision scheme which is at least convergent, can be transformed to a new scheme of arbitrarily high regularity.
We call $\mathcal{I}$ a \emph{smoothing operator} and $\tfrac{z+1}{2}$ a \emph{smoothing factor}. Note that the factor $z^{-1}$ in $\mathcal{I}$ is an index shift.
\begin{example}[B-Spline schemes]
The symbol of the scheme generating B-Spline curves of degree $\ell \geq 1$ and smoothness $C^{\ell-1}$ is
$$\ab_{\ell}^{\ast}(z)=\Big(\tfrac{(z+1)}{2}z^{-1}\Big)^{\ell}(z+1).$$
Obviously ${\ab}^{\ast}_{\ell}(z)=\tfrac{(z+1)}{2}z^{-1}{\ab}^{\ast}_{\ell-1}(z)=(\mathcal{I}{\ab}_{\ell-1})^{\ast}(z)$.
\end{example}

\section{Increasing the smoothness of vector subdivision schemes}
\label{sec:vector}
In this section we describe a procedure for increasing the smoothness of {\vs}s, which is similar to the scalar case. It is more involved since we consider masks consisting of matrix sequences.

\subsection{Convergence and smoothness analysis}
First we present results concerning the convergence and smoothness of {\vs}s. Their proofs can be found in \citet{charina05,cohen96,micchelli98,sauer02}.

For a mask $\A$ of a {\vs} we define
\begin{equation}\label{A0A1}
A^0=\sum_{i \in \m}A_{2i}, \quad A^1=\sum_{i \in \m}A_{2i+1}.
\end{equation}
Following \citet{micchelli98}, let
\begin{equation}\label{eigenspace}
\mathcal{E}_{\A}=\{v \in \mathbb{R}^p: A^0v=v \text{ and } A^1v=v\}
\end{equation}
and $k=\dim\mathcal{E}_{\A}$. A priori, $0 \leq k \leq p$. However, for a convergent {\vs}, $\mathcal{E}_{\A}\neq \{0\}$, i.e. $1 \leq k \leq p$. Therefore, the existence of a common eigenvector of $A^0$ and $A^1$ w.r.t.\ the eigenvalue $1$ is a necessary condition for convergence.

The next lemma reduces the convergence analysis to the case $\mathcal{E}_{\A}=\operatorname{span}\{e_1,\ldots,e_k\}.$
\begin{lemma}\label{Vconv}
Let $S_{\A}$ be a $C^{\ell}$ $(\ell \geq 0)$ convergent {\vs}. Given an invertible matrix $R \in \mathbb{R}^{p\times p}$, define a new mask $\hat{\A}$ by $\hat{A}_i=R^{-1}A_iR$ for $i \in \m$.
\begin{enumerate}
 \item  The {\vs} associated with $\hat{\A}$ is also $C^{\ell}$.
 \item \label{Rbasis} There exist invertible matrices such that $\hat{\A}$ satisfies $\mathcal{E}_{\hat{\A}}=\operatorname{span}\{e_1,\ldots,e_k\}$, where $k=\dim\mathcal{E}_{\A}$.
 \end{enumerate}
\end{lemma}
In \citet{cohen96,sauer02} the following generalization of the forward-difference operator $\D$ is introduced:
\begin{equation}\label{forwarddifference}
\D_k=\begin{pmatrix} \D I_k & 0 \\ 0 & I_{p-k} \end{pmatrix},
\end{equation}
where $I_k$ is the $(k\times k)$ unit matrix. It is shown there that if 
\begin{equation}\label{unitvectors}
\mathcal{E}_{\A}=\operatorname{span}\{e_1,\ldots,e_k\},
\end{equation}
then in analogy to \cref{scalarpartial}, there exists a matrix mask $\partial_{k}\A $ such that
\begin{equation}\label{derivedscheme}
\D_kS_{\A}=\tfrac{1}{2}S_{\partial_{k}\A}\D_k.
\end{equation}
Algebraic conditions guaranteeing \cref{derivedscheme} are stated and proved in the next subsection.

We denote by $\partial_k\A$ any mask satisfying \cref{derivedscheme}.
The vector scheme associated with $\partial_{k}\A$ is called the \emph{derived scheme} of $\A$ with respect to $\D_k$.
Furthermore, we have the following result concerning the convergence of $S_{\A}$ in terms of $S_{\partial_k\A}$:
\begin{theorem}\label{vectorsmooth} 
Let $\A$ be a mask such that $\mathcal{E}_{\A}=\operatorname{span}\{e_1,\ldots,e_k\}$. If $\|(\tfrac{1}{2}S_{\partial_{k}\A})^L\|<1$ for some $L \in \mathbb{N}$ (that is, $\tfrac{1}{2}S_{\partial_{k}\A}$ is contractive), then the vector scheme associated with $\A$ is convergent.
\end{theorem}
In fact there is a stronger result in \citet{charina05,cohen96}, but we only need this special case.
Two important results for the analysis of smoothness of {\vs}s are
\begin{theorem}[\citet{micchelli98}]
Let $\A$ be a mask of a convergent {\vs}, such that $\mathcal{E}_{\A}=\operatorname{span}\{e_1,\ldots,e_k\}$ for $k\leq p$, then
\begin{equation}
 \dim \mathcal{E}_{\partial_k \A}=\dim \mathcal{E}_{\A}.
\end{equation}
\end{theorem}
\begin{theorem}[\citet{charina05}]\label{smoothvector}
Let $\A$ be a mask such that $\mathcal{E}_{\A}=\operatorname{span}\{e_1,\ldots,e_k\}$. If the {\vs} associated with $\partial_{k}\A$ is $C^{\ell}$ for $\ell\geq 0$, then the {\vs} associated with $\A$ is $C^{\ell+1}$.
\end{theorem}
\begin{remark}
In the last theorem we omitted the assumption that $S_{\A}$ is convergent required in \citet{charina05}.This is possible because if $S_{\partial_k\A}$ is $C^{\ell}$,
then $\frac{1}{2}S_{\partial_k\A}$ is contractive implying that $S_{\A}$ is convergent in view of \Cref{vectorsmooth}.
\end{remark}
A useful observation for our analysis is
\begin{lemma}\label{eigen}
Let $\A$ be a matrix mask. Then
$$\mathcal{E}_{\A}=\{v \in \mathbb{R}^p: \A^{\ast}(1)v=2v \text{ and } \A^{\ast}(-1)v=0\}.$$ 
\end{lemma}
\begin{proof}
It follows immediately from \cref{A0A1} and the definition of a symbol that $A^0=\tfrac{1}{2}\Big(\A^{\ast}(1)+\A^{\ast}(-1)\Big)$ and $A^1=\tfrac{1}{2}\Big(\A^{\ast}(1)-\A^{\ast}(-1)\Big)$. This, together with \cref{eigenspace}, implies the claim of the lemma.
\end{proof}

\subsection{Algebraic conditions}
We would like to modify a given mask $\B$ of a $C^{\ell}$ {\vs} to obtain a new scheme $S_{\A}$ which is $C^{\ell+1}$.
The idea is to define $\A$ such that $\partial_k\A=\B$, i.e.\ such that \cref{derivedscheme} is satisfied for some $k$. If we can prove that $\mathcal{E}_{\A}=\operatorname{span}\{e_1,\ldots,e_k\}$,
then by \Cref{smoothvector}, the scheme $S_{\A}$ is $C^{\ell+1}$. There are some immediate questions:
 \begin{enumerate}
\item\label{que1} Under what conditions on a mask $\B$ can we define a mask $\A$ such that $\partial_k\A=\B$?
\item\label{que2} How to choose $k$?
\end{enumerate}
In order to answer these questions, we have to study in more details the mask of the derived scheme $\partial_k \A$ and its relation to the mask $\A$.
\begin{definition}\label{blockmatrixnot}
For a mask $\A$ of dimension $p$, i.e.\ $A_i \in \mathbb{R}^{p\times p}$ for $i\in \m$, and a fixed $k \in \{1,\ldots,p\},$ we introduce the block notation
\begin{equation*}
\A=\begin{pmatrix} \A_{11} & \A_{12} \\ \A_{21} & \A_{22}\end{pmatrix},
\end{equation*}
with $\A_{11}$ of size $(k \times k)$.
\end{definition}
In the next lemma, we present algebraic conditions on a symbol $\A^{\ast}(z)$ guaranteeing the existence of $\partial_k \A$ for a fixed $k \in \{1,\ldots,p\}$, and also show that if  
$\mathcal{E}_{\A}=\operatorname{span}\{e_1,\ldots,e_k\}$ these conditions hold.
\begin{lemma}\label{smooth1}
Let $\A,\B$ be masks of dimension $p$. With the notation of \Cref{blockmatrixnot} we have
\begin{enumerate}
\item\label{partial} If there exists $k \in \{1,\ldots,p\}$ such that $\A^{\ast}_{11}(-1)=0, \A^{\ast}_{21}(-1)=0$ and $\A^{\ast}_{21}(1)=0$, then there exists a mask $\partial_k\A$ satisfying \cref{derivedscheme}. 
\item\label{existsderived} If $\mathcal{E}_{\A}=\operatorname{span}\{e_1,\ldots,e_k\}$, then $\A^{\ast}(z)$ the conditions of (\ref{partial}) are satisfied.
\end{enumerate}
\end{lemma}
\begin{proof}
Under the assumptions of (\ref{partial}), the matrix
\begin{equation}\label{defD}
2\begin{pmatrix} \A_{11}^{\ast}(z)/(z^{-1}+1) & (z^{-1}-1)\A_{12}^{\ast}(z)\\[0.2cm]
                            \A_{21}^{\ast}(z)/(z^{-2}-1) &  \A^{\ast}_{22}(z)
                          \end{pmatrix},
\end{equation}
is a matrix Laurent polynomial. If we denote it by $(\partial_k \A)^{\ast}(z)$, then the equation $\D_kS_{\A}=\tfrac{1}{2}S_{\partial_k\A}\D_k$ is satisfied. Indeed, if we write this last equation in terms of symbols, we get
\begin{align} \label{equpartial}
&\begin{pmatrix} (z^{-1}-1)I_k & 0 \\ 0 & I_{p-k}\end{pmatrix}
\begin{pmatrix} \A^{\ast}_{11}(z) & \A^{\ast}_{12}(z) \\ \A^{\ast}_{21}(z) & \A^{\ast}_{22}(z)\end{pmatrix}\\\nonumber
&=
\begin{pmatrix} \A_{11}^{\ast}(z)/(z^{-1}+1) & (z^{-1}-1)\A_{12}^{\ast}(z)\\[0.2cm]
                            \A_{21}^{\ast}(z)/(z^{-2}-1) &  \A^{\ast}_{22}(z)
                          \end{pmatrix}
\begin{pmatrix} (z^{-2}-1)I_k & 0 \\ 0 & I_{p-k}\end{pmatrix}.
\end{align}
It is easy to verify the validity of \cref{equpartial}.

In order to prove (\ref{existsderived}), we deduce from \Cref{eigen} that $\mathcal{E}_{\A}=\operatorname{span}\{e_1,\ldots,e_k\}$ implies the properties of $\A$ required in (\ref{partial}).
\end{proof}

In the proof of the validity of our smoothing procedure for {\vs}s and {\hs}s, we work with the algebraic conditions (\ref{partial}) in \Cref{smooth1} rather than with assumption (\ref{unitvectors}).
The reason is that the algebraic conditions can be checked and handled more easily. 

In order to define a procedure for increasing the smoothness of {\vs}s, we start by answering question (\ref{que1}):
\begin{lemma}\label{smooth2}
 Let $\A,\B$ be masks of dimension $p$ and let $k\in \{1,\ldots,p\}$. With the notation of \Cref{blockmatrixnot}, if
 $\B_{12}^{\ast}(1)=0$, then there exists a mask $\mathcal{I}_k \B$ satisfying
\begin{equation}\label{eqI}
\D_kS_{\mathcal{I}_k\B}=\tfrac{1}{2}S_{\B}\D_k,
\end{equation}
where $\D_k$ is defined in \cref{forwarddifference}.
\end{lemma}
\begin{proof}
 Defining 
\begin{equation}\label{defI}
(\mathcal{I}_k \B)^{\ast}(z)=\frac{1}{2}\begin{pmatrix} (z^{-1}+1)\B_{11}^{\ast}(z) & \B_{12}^{\ast}(z)/(z^{-1}-1)\\
                            (z^{-2}-1)\B_{21}^{\ast}(z) &  \B^{\ast}_{22}(z)
                          \end{pmatrix},
\end{equation}
we note that under the condition $\B_{12}^{\ast}(1)=0$, the above matrix is a matrix Laurent polynomial. It is easy to verify that the matrix $(\mathcal{I}_k \B)^{\ast}(z)$ in \cref{defI} satisfies \cref{eqI}.
\end{proof}
\begin{remark}\label{trivialcase}
If $k=p$ in \Cref{smooth2} then $(\mathcal{I}_p\B)^{\ast}(z)=\tfrac{z^{-1}+1}{2}\B^{\ast}(z)$, where $\tfrac{z^{-1}+1}{2}$ is the smoothing factor in the scalar case.
\end{remark}
In \Cref{smooth1} and \Cref{smooth2} we constructed two operators
$\partial_k$ and $\mathcal{I}_k$ operating on masks, which (under some conditions) are inverse to each other. Denote by 
$\ell_a^k$ the set of all masks satisfying the conditions (\ref{partial}) of \Cref{smooth1} and by
$\ell_b^k$ the set of all masks satisfying the condition of \Cref{smooth2}.
Then it is easy to show that
\begin{align}
\partial_k: \quad &\ell_a^k \to \ell_b^k \hspace{2cm}  \mathcal{I}_k: \quad \ell_b^k \to \ell_a^k
\end{align}
and that 
\begin{equation}
\partial_k(\mathcal{I}_k \B)=\B \quad \text{and} \quad \mathcal{I}_k(\partial_k \A)=\A.
\end{equation}
This shows that the condition of \Cref{smooth2} on a mask $\B$ allows to define a mask $\A=\mathcal{I}_k \B$ such that $\partial_k\A=\B$. This anwers question (\ref{que1}). Still we need to deal with question (\ref{que2}).
\begin{remark}
It follows from \Cref{smooth1} and \Cref{smooth2} that the existence of $\partial_k \A$ and $\mathcal{I}_k \B$ depends only on algebraic conditions. Yet this is not sufficient to define a procedure for changing the mask of a {\vs}
in order to get a mask associated with a smoother \vs.
Even if $\mathcal{I}_k\B$ exists for some $k$, the application of \Cref{smoothvector}, in view of \Cref{Vconv}, to $\A=\mathcal{I}_k\B$ is based on the dimension of $\mathcal{E}_{\A}$ which is not necessarily $k$.
But if $\mathcal{E}_{\A}=\operatorname{span}\{e_1,\ldots,e_k\}$, we can conclude from \Cref{smoothvector} that $S_{\A}$ has smoothness increased by $1$ compared to the smoothness of $S_{\B}$.

In the next section we show that if for $\B$ associated with a converging {\vs} $\dim\mathcal{E}_{\A}=k$, then there exists a canonical transformation $\overline{R}$ such that $\overline{\B}=\overline{R}^{-1}\B \overline{R}$
satisfies the algebraic conditions of \Cref{smooth2} and
$\mathcal{E}_{\mathcal{I}_k\B}=\operatorname{span}\{e_1,\ldots,e_k\}$. Therefore by \Cref{smoothvector}, if $S_{\B}$ is $C^{\ell}$, then $S_{\mathcal{I}_k\overline{\B}}$ is $C^{\ell+1}$.
\end{remark}
\subsection{The canonical transformations to the standard basis}\label{sec:transform}
Let $\B$ be a mask of a convergent {\vs} $S_{\B}$. Denote by $k=\dim\mathcal{E}_{\B}$. We define a new mask $\overline{\B}$ such that
\begin{equation}\label{3properties}
\mathcal{E}_{\overline{\B}}=\operatorname{span}\{e_1,\ldots,e_k\}, \, \overline{\B} \in \ell^k_b \text{ and } \mathcal{E}_{\mathcal{I}_k\overline{\B}}=\operatorname{span}\{e_1,\ldots,e_k\}.
\end{equation}
This is achieved by considering the matrix $M_{\B}=\tfrac{1}{2}(B^0+B^1)$. First we state a result of importance to our analysis, which follows from Theorem 2.2 in \citet{cohen96} and from its proof.
\begin{theorem}\label{eigenvaluesmod}
Let $\B$ be a mask of a convergent {\vs}. A basis of $\mathcal{E}_{\B}$ is also a basis of the eigenspace of $M_{\B}=\tfrac{1}{2}(B^0+B^1)$ corresponding to the
eigenvalue $1$. Moreover $\lim_{n \to \infty}M_{\B}^n$ exists.
\end{theorem}
A direct consequence of the last theorem, concluded from the existence of $\lim_{n \to \infty}M_{\B}^n$, is
\begin{corollary}\label{alggeomulti}
Let $\B$ be a mask associated with a converging {\vs}. Then the algebraic multiplicity of the eigenvalue $1$ of $M_{\B}$ equals its geometric multiplicity, and all its other eigenvalues have modulus less than $1$.
\end{corollary}
In particular, since $M_{\B}=\frac{1}{2}\B^{\ast}(1)$, \Cref{eigenvaluesmod} implies that if $S_{\B}$ is a convergent {\vs}, then $\mathcal{E}_{\B}$ is the eigenspace of $B^{\ast}(1)$ w.r.t.\ to the eigenvalue $2$.

We proceed to define from a mask $\B$ associated with a convergent {\vs}, a new mask $\overline{\B}$ satisfying \cref{3properties}.

Let $\B$ be a mask associated with a convergent {\vs} and let $\mathcal{V}=\{v_1,\ldots,v_k\}$ be a basis of $\mathcal{E}_{\B}$ (and therefore also a basis of the eigenspace w.r.t.\ $1$ of $M_{\B}$).
We define a real matrix 
\begin{equation}\label{R}
\overline{R}=[v_1, \ldots, v_k | Q],
\end{equation}
where the columns of $Q$ span the invariant space of $M_{\B}$ corresponding to the eigenvalues different from 1 of $M_{\B}$.
$Q$ completes $\mathcal{V}$ to a basis of $\mathbb{R}^p$ and $\overline{R}$ is an invertible matrix. We call $\overline{R}$ defined by \cref{R} a canonical transformation. There are many canonical transformations, since $Q$ is not unique. 
Our smoothing procedure is independent of the choice of a canonical transformation.
Define a modified mask $\overline{\B}$ by
\begin{equation}\label{bartransform}
\overline{B}_i=\overline{R}^{-1}B_i\overline{R}, \quad \text{for } i \in \m.
\end{equation}
Then by \cref{R} and \Cref{eigenvaluesmod} we have that $\mathcal{E}_{\overline{\B}}=\operatorname{span}\{e_1,\ldots,e_k\}$. This proves the first claim in \cref{3properties}. Also by \Cref{Vconv}, $S_{\overline{\B}}$
is convergent and has the same smoothness as $S_{\B}$.

Furthermore, by \cref{R},
\begin{equation}\label{jordan}
M_{\overline{\B}}=\tfrac{1}{2}(\overline{B}^0+\overline{B}^1)=\overline{R}^{-1}M_{\B}\overline{R}=\begin{pmatrix}
I_k & 0 \\
0 & J
\end{pmatrix}
\end{equation}
is the Jordan form of $M_{\B}$. By \Cref{alggeomulti}, $J$ has eigenvalues with modulus less than $1$. Transformations $\overline{R}$ which result in representations of $M_{\B}$ similar to the one in \cref{jordan} have already been considered in e.g.\
\citet{cohen96,sauer02}.
The special structure of $M_{\overline{\B}}$ is the key to our smoothing procedure. The next theorem follows from \cref{jordan} and proves the remaining claims of \cref{3properties}.
\begin{theorem}\label{propertiesbar}
Let $S_{\B}$ be a convergent {\vs} and let $k=\dim \mathcal{E}_{\B}$. Define $\overline{\B}$ by \cref{bartransform} with $\overline{R}$ a canonical transformation. Then $\overline{\B}$ has the following properties:
\begin{enumerate}
\item\label{propb} $\overline{\B} \in \ell^k_b$,
\item\label{smoothedeigen} $\mathcal{E}_{\mathcal{I}_k\overline{\B}}=\operatorname{span}\{e_1,\ldots,e_k\}$.
\end{enumerate}
\end{theorem}
\begin{proof}
We start by proving (\ref{propb}).
Since
\begin{equation}\label{Bast1}
\overline{\B}^{\ast}(1)=\overline{B}^0+\overline{B}^1=2M_{\overline{\B}}=2\begin{pmatrix}
I_k & 0 \\
0 & J
\end{pmatrix},
\end{equation}
it follows that $\overline{\B}^{\ast}_{12}(1)=0$. Thus by \Cref{smooth2}, $\overline{\B} \in \ell^k_b$ and therefore
$\mathcal{I}_k\overline{\B}$ exists.

In order to prove (\ref{smoothedeigen}), we use \Cref{eigen} and show that $\mathcal{E}_{\mathcal{I}_k\overline{\B}}=\{v\in \mathbb{R}^p:(\mathcal{I}_k\overline{\B})^{\ast}(1)v=2v \text{ and } \allowbreak
(\mathcal{I}_k\overline{\B})^{\ast}(-1)v=0 \}$ is spanned by $e_1,\ldots,e_k$. Indeed by \cref{Bast1} it follows that $\overline{\B}^{\ast}_{11}(1)=2I_k$ and
$\overline{\B}^{\ast}_{22}(1)=2J$.
Since by \cref{Bast1} $\overline{\B}^{\ast}_{12}(1)=0$, there exists a symbol $\mathbf{C}^{\ast}(z)$ such that $\overline{\B}^{\ast}_{12}(z)=(z^{-1}-1)\mathbf{C}^{\ast}(z)$, and therefore \cref{defI} implies the block form:
\begin{equation}\label{IBbar}
(\mathcal{I}_k\overline{\B})^{\ast}(1)=\begin{pmatrix} 2I_k & \tfrac{1}{2}\mathbf{C}^{\ast}(1) \\[0.2cm]
							0 & J \end{pmatrix}, \quad
(\mathcal{I}_k\overline{\B})^{\ast}(-1)=\begin{pmatrix}0 & \tfrac{1}{2}\mathbf{C}^{\ast}(-1) \\[0.2cm]
							0 & \tfrac{1}{2}\overline{\B}^{\ast}_{22}(-1) \end{pmatrix},
\end{equation}
\Cref{IBbar}, in view of \Cref{eigen}, implies that $\operatorname{span}\{e_1,\ldots,e_k\}=\mathcal{E}_{\mathcal{I}_k\overline{\B}}$, since the eigenspace of $(\mathcal{I}_k\overline{\B})^{\ast}(1)$ w.r.t.\ the eigenvalue $2$ is exactly $\operatorname{span}\{e_1,\ldots,e_k\}$
(the matrix $J$ only contributes eigenvalues with modulus less than $1$), and these vectors are in the kernel of $(\mathcal{I}_k\overline{\B})^{\ast}(-1)$.
\end{proof}

Summarizing the above results, we arrive at
\begin{corollary}\label{cor:tomas}
 Let $\B$ be a mask of a convergent {\vs}, let $k=\dim \mathcal{E}_{\B}$, and let $\overline{\B}$ be as in \Cref{propertiesbar}. Then $\mathcal{I}_k\overline{\B}$ exists and
 \begin{equation*}
   \mathcal{E}_{\mathcal{I}_k\overline{\B}}= \mathcal{E}_{\overline{\B}}=\operatorname{span}\{e_1,\ldots,e_k\}.
 \end{equation*}
\end{corollary}
\subsection{A procedure for increasing the smoothness}
\Cref{propertiesbar} allows us to define the following procedure which generates {\vs}s of higher smoothness from given convergent {\vs}s:
\begin{alg}\label{algorithm}
The input data is a mask $\B$ associated with a $C^{\ell}$ {\vs}, $\ell \geq 0$, and the output is a mask $\A$ associated with a $C^{\ell+1}$ {\vs}.
\begin{enumerate}
\item\label{0} Choose a basis $\mathcal{V}$ of $\mathcal{E}_{\B}$ and define $\overline{R}$, a canonical transformation, as in \cref{R}.
\item\label{1} Define $\overline{\B}=\overline{R}^{-1}\B \overline{R}$.
\item Define $k=\operatorname{dim}(\mathcal{E}_{\B})$.
\item\label{2} Define $\overline{\A}=\mathcal{I}_k\overline{\B}$ as in \cref{defI}.
\item\label{3} Define $\A=\overline{R}\,\overline{\A}\, \overline{R}^{-1}$.
\end{enumerate}
\end{alg}
A schematic representation of \Cref{algorithm} is given in \Cref{figvectorsmooth}.
\begin{remark}\label{rem:span}
Step \ref{3} in \Cref{algorithm} is not essential. The scheme $S_{\overline{\A}}$ is already $C^{\ell+1}$. Step \ref{3} guarantees that $\mathcal{E}_{\A}=\mathcal{E}_{\B}$.
In both cases to apply another smoothing procedure to get a $C^{\ell+2}$ {\vs}, a new canonical transformation has to be applied.
\end{remark}
In the notation of \Cref{algorithm}, we define the smoothing operator $\mathcal{I}_k$ applied to a mask $\B$ of a convergent {\vs} as
\begin{equation}\label{vec:soperator}
 \mathcal{I}_k\B=\overline{R}(\mathcal{I}_k\overline{\B})\overline{R}^{-1}.
\end{equation}
This is a generalization of the smoothing operator in the case of scalar subdivision schemes.

An important property of \Cref{algorithm}, which is easily seen from \cref{defI} is,
\begin{corollary}\label{vectorsupport}
Assume that $\B$ and $\A$ are masks as in \Cref{algorithm}. If the support of $\B$ is contained in $[-N_1,N_2]$ with $N_1,N_2 \in \mathbb{N}$,
then the support of $\A$ is contained in $[-N_1-2,N_2]$.
\end{corollary}
Therefore \Cref{algorithm} increases the support length by at most $2$, independently of the dimension of the mask. Recall that in the scalar case the support size is increased by 1.

An interesting observation follows from \Cref{algorithm}, \cref{Bast1} and \cref{IBbar},
\begin{corollary}\label{cor:eigenvalues}
Assume that $\A, \B$ are masks as in \Cref{algorithm}. 
Then $\A^{\ast}(1)$ and $\B^{\ast}(1)$ share the eigenvalue $2$ and the corresponding eigenspace. To each eigenvalue $\lambda \neq 2$ of $\B^{\ast}(1)$ there is an eigenvalue $\frac{1}{2}\lambda$ of $\A^{\ast}(1)$.
\end{corollary}
Note that a similar result to that in \Cref{cor:eigenvalues} is in general not true for $\B^{\ast}(-1)$ and $\A^{\ast}(-1)$. However, \Cref{ex:doubleknot} shows that this can well be the case.

\begin{figure}
\centering
\begin{tikzpicture}[scale=4.5]
\node (A) at (-0.1,1) {$S_{\B} \in C^{\ell}$};
\node (B) at (1.4,1) {$S_{\overline{\B}} \in C^{\ell}$};
\node (C) at (-0.1,0.45) {$S_{\A} \in C^{\ell+1}$};
\node (D) at (1.4,0.45) {$S_{\overline{\A}} \in C^{\ell+1}$};
\node (E) at (-0.2,0.97) {\phantom{a}};
\path[->,font=\scriptsize]
(A) edge node[above]{$\overline{\B}=R^{-1}\B R$} (B)
(B) edge node[right]{$\overline{\A}=\mathcal{I}_k\overline{\B}$} (D)
(D) edge node[above]{$\A=R\overline{\A}R^{-1}$} (C);
\end{tikzpicture}
\caption{A schematic representation of \Cref{algorithm}.}
\label{figvectorsmooth}
\end{figure}
\begin{example}[Double-knot cubic spline subdivision]\label{ex:doubleknot}
We consider the {\vs} with symbol
\begin{equation}\label{ex:mask}
\B^{\ast}(z)=\frac{1}{8}\begin{pmatrix}
2+6z+z^2 & 2z+5z^2 \\
5+2z & 1+6z+2z^2
\end{pmatrix}.
\end{equation}
It is known that this scheme produces $C^1$ limit curves (see e.g.\ \citet{dyn02}). We apply \Cref{algorithm} to $\B$ to obtain a {\vs} $S_{\A}$ of regularity $C^2$:
\begin{enumerate}
\item\label{a} First we find a basis of $\mathcal{E}_{\B}$ in order to compute a canonical transformation $\overline{R}$.
The matrices $\B^{\ast}(1)$ and $\B^{\ast}(-1)$ are given by
\begin{equation*}
\B^{\ast}(1)=\frac{1}{8}\begin{pmatrix}
9 & 7 \\
7 & 9
\end{pmatrix}, \quad
\B^{\ast}(-1)=\frac{1}{8}\Big(\begin{array}{r r}
-3 & 3 \\
3 & -3
\end{array}\Big)
\end{equation*}
and have the following eigenvalues and eigenvectors
\begin{align}\label{ex:equ}
& \text{For } \B^{\ast}(1): \quad \text{eigenvalues}: 2, \tfrac{1}{4},\quad  \text{eigenvectors: }\Big(\begin{array}{r r} 1 \\ 1 \end{array}\Big), \Big(\begin{array}{r r} -1 \\ 1 \end{array}\Big), \text{ resp.\ }\\ \nonumber
& \text{For } \B^{\ast}(-1):  \quad \text{eigenvalues}: 0, -\tfrac{3}{4}, \quad  \text{eigenvectors: }\Big(\begin{array}{r r} 1 \\ 1 \end{array}\Big), \Big(\begin{array}{r r} -1 \\ 1 \end{array}\Big), \text{ resp.\ }
\end{align}
Therefore $\mathcal{E}_{\B}$ is spanned by ${1 \choose 1}$. The transformation $\overline{R}$ is determined by the eigenvectors of $\B^{\ast}(1)$:
\begin{equation*}
R=\Big(\begin{array}{r r}
1 & -1 \\
1 & 1
\end{array}\Big), \quad
R^{-1}=\tfrac{1}{2}\Big(\begin{array}{r r}
1 & 1 \\
-1 & 1
\end{array}\Big).
\end{equation*}
\item We continue by computing $\overline{\B}=\overline{R}^{-1}\B \overline{R}$ from the symbol of $\B$ in \cref{ex:mask}, and get
\begin{equation*}
\overline{\B}^{\ast}(z)=\frac{1}{8}\begin{pmatrix}
4(1+z)^2 & 3(z^2-1) \\
-2(z^2-1) & -1+4z-z^2
\end{pmatrix}.
\end{equation*}
\item From Step \ref{a} we see that $k=\dim \mathcal{E}_{\B}=1$.
\item We compute $\overline{\A}=\mathcal{I}_1\overline{\B}$ by computing its symbol.
\begin{equation*}
\overline{\A}^{\ast}(z)=\frac{1}{16}\begin{pmatrix}
4z^{-1}(1+z)^3 & -3z^{-1}(z+1) \\
2z^{-2}(z^2-1)^2 & -1+4z-z^2
\end{pmatrix}.
\end{equation*}
\item In this step we transform back to the original basis $\A=\overline{R}\,\overline{\A}\,\overline{R}^{-1}$, by deriving $\A^{\ast}(z)$.
\begin{equation}\label{ex:vectA}
\A^{\ast}(z)=\frac{1}{32}z^{-2}\begin{pmatrix}
z^4+16z^3+18z^2+7z-2 & \enspace 3z^4+8z^3+14z^2+z-2 \\[0.3cm]
7z^4+8z^3+12z^2+7z+2 & \enspace 5z^4+16z^3+4z^2+z+2
\end{pmatrix}.
\end{equation}
\end{enumerate}
It follows from the analysis preceeding \Cref{algorithm} that $S_{\A}$ is $C^2$.

To verify \Cref{rem:span} we show that $\mathcal{E}_{\A}$ has the same basis as $\mathcal{E}_{\B}$. We compute 
\begin{equation*}
\A^{\ast}(1)=\frac{1}{8}\Big(\begin{array}{r r}
10 & 6 \\[0.1cm]
9 & 7
\end{array}\Big), \quad
\A^{\ast}(-1)=\frac{1}{16}\Big(\begin{array}{r r}
-3 & 3 \\
3 & -3
\end{array}\Big)
\end{equation*}
and their eigenvalues and eigenvectors:
\begin{align}\label{ex:equ2}
&\text{For } \A^{\ast}(1):  \quad \text{eigenvalues}: 2, \tfrac{1}{8},\quad  \text{eigenvectors: }\Big(\begin{array}{r r} 1 \\ 1 \end{array}\Big), \Big(\begin{array}{r r} -2 \\ 3 \end{array}\Big), \text{ resp.\ }\\ \nonumber
&\text{For } \A^{\ast}(-1):  \quad \text{eigenvalues}: 0, -\tfrac{3}{8}, \quad  \text{eigenvectors: }\Big(\begin{array}{r r} 1 \\ 1 \end{array}\Big), \Big(\begin{array}{r r} -1 \\ 1 \end{array}\Big),\text{ resp.\ }
\end{align}
Therefore by \cref{ex:equ}, \cref{ex:equ2} and \Cref{eigen}, $\mathcal{E}_{\A}$ and $\mathcal{E}_{\B}$ are spanned by ${1 \choose 1}$.

Note that the eigenvectors corresponding to the eigenvalues which have modulus less than $1$, of $M_{\A}$ and $M_{\B}$ are different. Thus in order to generate a $C^3$ scheme from $S_{\A}$, a new canonical transformation has to be computed.

Also, comparing the eigenvalues of $\A^{\ast}(1)$ and $\B^{\ast}(1)$ we see that \Cref{cor:eigenvalues} is satisfied. In fact in this example, also the eigenvalues of $\A^{\ast}(-1)$ and $\B^{\ast}(-1)$ have the same property. 

It is easy to see from \cref{ex:vectA} that the support of the mask $\A$ is 4, and from \cref{ex:mask} that the support of $\B$ is 2, in accordance with \Cref{vectorsupport}.
\end{example}

\section{Increasing the smoothness of Hermite subdivision schemes}\label{sec:hermite}
In this section we describe a procedure for increasing the smoothness of {\hs}s refining function and first derivative values, based on the procedure for the vector case described in \Cref{sec:vector}.
We consider {\hs}s which operate on data $\bc \in \ell(\mathbb{R}^2)$, using the notation of \Cref{background}.

\subsection{Algebraic conditions}
As in the vector case, {\hs}s use matrix-valued masks $\A=\{A_i \in \mathbb{R}^{2\times 2}: i\in \m \}$ and subdivision operators $S_{\A}$ as defined in \cref{matrixoperator}. 
The input data $\bc^0\in \ell(\mathbb{R}^2)$ is refined via $D^n\bc^n=S_{\A}^n\bc^0$, where $D$ is the dilation matrix
\begin{equation*}
D=\begin{pmatrix} 1 & 0 \\ 0 & \frac{1}{2} \end{pmatrix}.
\end{equation*}
A {\hs} is called \emph{interpolatory} if its mask $\A$ satisfies $A_0=D$ and $A_{2i}=0$ for all $i \in \m \backslash \{0\}$.

We always assume that a {\hs} satisfies the \emph{spectral condition} \citep{dubuc09}. This condition requires that there is $\varphi \in \mathbb{R}$ such that both the constant sequence 
${\mathbf{k}}=\{\bigl(\begin{smallmatrix} 1 \\ 0 \end{smallmatrix}\bigr): i \in \m\}$
and the linear sequence
${\Bell}=\{\bigl(\begin{smallmatrix} i +\varphi \\ 1 \end{smallmatrix}\bigr): i \in \m \}$ obey the rule
\begin{equation}\label{spectral}
S_{\A}{\mathbf{k}}={\mathbf{k}}, \quad S_{\A}{\Bell}=\tfrac{1}{2}{ {\Bell}}.
\end{equation}
The spectral condition is crucial for the convergence and smoothness analysis of linear {\hs}s.
If the {\hs} is interpolatory we can choose $\varphi=0$.

We now characterize the spectral condition in terms of the symbol of the mask $\A$. We introduce the notation
\begin{equation}\label{blockmatrix}
\A=\begin{pmatrix}\ab_{11} & \ab_{12} \\ \ab_{21} & \ab_{22}\end{pmatrix},
\end{equation}
where $\ab_{ij} \in \ell(\mathbb{R})$ for $i,j \in \{1,2\}$.
It is easy to verify that the spectral condition in \cref{spectral} is equivalent to the algebraic conditions in the next lemma.
\begin{lemma}\label{spectralcond}
A mask $\A$ satisfies the spectral condition given by \cref{spectral} with $\varphi \in \mathbb{R}$ if and only if its symbol $\A^{\ast}(z)$ satisfies
\begin{enumerate}
\item\label{spectral11} $\ab^{\ast}_{11}(1)=2,\: \ab^{\ast}_{11}(-1)=0$.
\item\label{spectral21} $\ab^{\ast}_{21}(1)=0,\: \ab^{\ast}_{21}(-1)=0$.
\item\label{spectral11mix} ${\ab^{\ast}_{11}}'(1)-2\ab^{\ast}_{12}(1)=2\varphi,\: {\ab^{\ast}_{11}}'(-1)+2\ab^{\ast}_{12}(-1)=0$.
\item\label{spectral21mix} ${\ab^{\ast}_{21}}'(1)-2\ab^{\ast}_{22}(1)=-2,\: {\ab^{\ast}_{21}}'(-1)+2\ab^{\ast}_{22}(-1)=0.$
\end{enumerate}
\end{lemma}
Parts (\ref{spectral11}) and (\ref{spectral21}) relate to the reproduction of constants, whereas parts (\ref{spectral11mix}) and (\ref{spectral21mix}) are related to the reproduction of linear functions.

Next we cite results on $HC^{\ell}$ smoothness of {\hs}.
Consider the \emph{Taylor operator} $T$, first introduced in \citet{merrien12}:
\begin{equation*}
T=\Big(\begin{array}{r r} \Delta & -1 \\ 0 & 1 \end{array}\Big).
\end{equation*}
The Taylor operator is a natural analogue of the operator $\D_k$ for {\vs}s and the forward difference operator $\D$ in scalar subdivision. We have the following result analogous to \cref{derivedscheme}:
\begin{lemma}[\citet{merrien12}]\label{existtaylor}
If the {\hs} associated with a mask $\A$ satisfies the spectral condition of \cref{spectral}, then there exists a matrix mask of dimension $2$, $\partial_{t}\A$, such that
\begin{equation}\label{derivedtaylor}
TS_{\A}=\tfrac{1}{2}S_{\partial_t\A}T.
\end{equation}
The mask $\partial_t \A$ determines a {\vs} called the \emph{Taylor scheme} associated with ${\A}$.
\end{lemma}

\subsection{Properties of the Taylor scheme}\label{sec:propertiestaylor}
In order to increase the smoothness of a {\hs}, the obvious idea is to pass to its Taylor scheme defined in \cref{derivedtaylor}, increase the smoothness of this {\vs} by \Cref{algorithm}
and then use the resulting {\vs} as the Taylor scheme of a new {\hs}.
The first question which arises in this process is if the last step is always possible, i.e., if the smoothing operator $\mathcal{I}_k$ of \cref{vec:soperator} maps Taylor schemes to Taylor schemes.
To answer this question depicted in \Cref{fighermite1}, we state algebraic conditions on a mask $\B$ of a {\vs} guaranteeing that $S_{\B}$ is a Taylor scheme.
\begin{definition}\label{taylor}
The algebraic conditions on a mask $\B$,
\begin{enumerate}
\item\label{tmask1} $\bb_{12}^{\ast}(1)=0, \bb_{12}^{\ast}(-1)=0,$
\item\label{tmask2} $\bb_{22}^{\ast}(1)=2, \bb_{22}^{\ast}(-1)=0,$
\item\label{tmask3} $\bb_{11}^{\ast}(1)+\bb_{21}^{\ast}(1)=2,$
\end{enumerate}
are called \emph{Taylor conditions}.
(Here we use the notation of \cref{blockmatrix}).
\end{definition}
We prove in \Cref{masklemma1} that the mask $\partial_t \A$ obtained via \cref{derivedtaylor} satisfies the Taylor conditions. This justifies the name \emph{Taylor conditions}.
\begin{remark}\label{rem:e2}
It is easy to verify that conditions (\ref{tmask1}) and (\ref{tmask2}) of \Cref{taylor} are equivalent to $e_2 \in \mathcal{E}_{\B}$.
\end{remark}

\begin{figure*}[!b]
\centering
\begin{tikzpicture}[scale=4.5]
\node (A) at (-0.1,1) {$S_{\A} \in HC^{\ell}$};
\node (B) at (1.4,1) {$S_{\partial_t\A} \in C^{\ell-1}$};
\node (C) at (-0.1,0.45) {$S_{\C} \in HC^{\ell+1}$};
\node (D) at (1.4,0.45) {$S_{\mathcal{I}_k\partial_t\A} \in C^{\ell}$};
\node (E) at (-0.2,0.97) {\phantom{a}};
\path[->,font=\scriptsize]
(A) edge node[above]{$\partial_t$} (B)
(B) edge node[right]{$\mathcal{I}_k$} (D)
(D) edge node[above]{?} (C);
\end{tikzpicture}
\caption{A schematic representation of the idea for smoothing {\hs}s.}
\label{fighermite1}
\end{figure*}

The next lemmas are concerned with the connection between masks satisfying the spectral condition of \cref{spectral} and masks satisfying the Taylor conditions of \Cref{taylor}.
\begin{lemma}\label{masklemma1}
Let $\A$ be a mask satisfying the spectral condition. Then we can define a mask $\partial_t\A$ such that \cref{derivedtaylor} is satisfied,
and $\partial_t\A$ satisfies the Taylor conditions.
\end{lemma}
Note that the existence of $\partial_t\A$ in \Cref{masklemma1} is a result of \citet{merrien12} (see \Cref{existtaylor}). We prove it here because its proof is used in our analysis.
\renewcommand{\bc}{\boldsymbol{\gamma}}
\begin{proof}
By solving \cref{derivedtaylor} in terms of symbols for $\partial_t\A$, it is easy to see that
\begin{align}\label{deriv11}
(\partial_t\A)^{\ast}_{11}(z)&=2\Big(\frac{\ab_{11}^{\ast}(z)}{z^{-1}+1}-\frac{\ab_{21}^{\ast}(z)}{z^{-2}-1}\Big),\\ \label{deriv12}
(\partial_t\A)^{\ast}_{12}(z)&=2\Big((z^{-1}-1)\ab_{12}^{\ast}(z)-\ab_{22}^{\ast}(z)+\frac{\ab_{11}^{\ast}(z)}{z^{-1}+1}-\frac{\ab_{21}^{\ast}(z)}{z^{-2}-1}\Big),\\ \label{deriv21}
(\partial_t\A)^{\ast}_{21}(z)&=2\frac{\ab_{21}^{\ast}(z)}{z^{-2}-1},\\ \label{deriv22}
(\partial_t\A)^{\ast}_{22}(z)&=2\Big(\ab_{22}^{\ast}(z)+ \frac{\ab_{21}^{\ast}(z)}{z^{-2}-1}\Big). 
\end{align}
By the algebraic conditions of \Cref{spectralcond}, $(\partial_t\A)^{\ast}(z)$ defined by eqs. (\ref{deriv11}) -- (\ref{deriv22}) is a Laurent polynomial.
Note that we only need the first two conditions of \Cref{spectralcond} equivalent to the reproduction of constants to define $\partial_t\A$.

We now show that $\partial_t\A$ satisfies the Taylor conditions. Multiplying \cref{deriv12} with the factor $(z^{-2}-1)$, differentiating with respect to $z$, substituting $z=1$ and $z=-1$, and applying \Cref{spectralcond}, we obtain:
\begin{align*}
(\partial_t \A)^{\ast}_{12}(1)&=-2\ab_{22}^{\ast}(1)+\ab_{11}^{\ast}(1)+{\ab_{21}^{\ast}}'(1)=0, \\ 
(\partial_t \A)^{\ast}_{12}(-1)&=-4\ab_{12}^{\ast}(-1)-2\ab_{22}^{\ast}(-1)-2{\ab_{11}^{\ast}}'(-1)-\ab_{11}^{\ast}(-1)-{\ab_{21}^{\ast}}'(-1)=0.
\end{align*}
This proves that part (\ref{tmask1}) of \Cref{taylor} is satisfied.

Applying the same procedure to \cref{deriv22}, we obtain
\begin{align*}
(\partial_t \A)^{\ast}_{22}(1)&=2\ab^{\ast}_{22}(1)-{\ab_{21}^{\ast}}'(1)=2,\\
(\partial_t \A)^{\ast}_{22}(-1)&=2\ab^{\ast}_{22}(-1)+{\ab_{21}^{\ast}}'(-1)=0.\\
\end{align*}
This concludes part (\ref{tmask2}) of \Cref{taylor}. Similarly \cref{deriv11,deriv21} imply
\begin{equation*}
(\partial_t \A)^{\ast}_{11}(1)+(\partial_t \A)^{\ast}_{21}(1)=(2+{\ab_{21}^{\ast}}'(1))-{\ab_{21}^{\ast}}'(1)=2,
\end{equation*}
which proves (\ref{tmask3}) of \Cref{taylor}.
\end{proof}

\begin{lemma}\label{masklemma2}
 Let $\B$ be a mask satisfying the Taylor conditions. Then we can define a mask $\mathcal{I}_t\B$ such that 
 $$TS_{\mathcal{I}_t\B}=\tfrac{1}{2}S_{\B}T$$
 is satisfied, and $\mathcal{I}_t\B$ satisfies the spectral condition.
\end{lemma}

\begin{proof}
Suppose that $\B$ satisfies the Taylor conditions. We define a mask $\mathcal{I}_t\B$ satisfying the equation $TS_{\mathcal{I}_t\B}=\tfrac{1}{2}S_{\B}T$ by writing it in terms of symbols. This yields the symbol
\begin{align}\nonumber
(\mathcal{I}_t\B)^{\ast}_{11}(z)=& \tfrac{1}{2}(z^{-1}+1)(\bb^{\ast}_{11}(z)+\bb^{\ast}_{21}(z)),\\ \label{int12}
(\mathcal{I}_t\B)^{\ast}_{12}(z)=& \tfrac{1}{2}\Big(\bb^{\ast}_{12}(z)-\bb^{\ast}_{11}(z)-\bb^{\ast}_{21}(z)+\bb^{\ast}_{22}(z)\Big)\Big/(z^{-1}-1),\\ \nonumber
(\mathcal{I}_t\B)^{\ast}_{21}(z)=& \tfrac{1}{2}\bb^{\ast}_{21}(z)(z^{-2}-1),\\ \nonumber
(\mathcal{I}_t\B)^{\ast}_{22}(z)=& \tfrac{1}{2}(\bb^{\ast}_{22}(z)-\bb^{\ast}_{21}(z)),
\end{align}
It follows from the Taylor conditions that $(\mathcal{I}_t\B)^{\ast}(z)$ is a Laurent polynomial and thus well-defined.

We continue by showing that $\mathcal{I}_t\B$ satisfies the spectral condition.
It is immediately clear from the definition of $\mathcal{I}_t\B$ that (\ref{spectral11}) and (\ref{spectral21}) of \Cref{spectralcond} are satisfied. Furthermore, it is easy to see that
\begin{align*}
&{(\mathcal{I}_t\B)^{\ast}_{21}}'(1)-2(\mathcal{I}_t\B)^{\ast}_{22}(1)=-\bb_{21}^{\ast}(1)
-\bb_{22}^{\ast}(1)+\bb_{21}^{\ast}(1)=-2,\\
&{(\mathcal{I}_t\B)^{\ast}_{21}}'(-1)+2(\mathcal{I}_t\B)^{\ast}_{22}(-1)=\bb_{21}^{\ast}(1)
+\bb_{22}^{\ast}(-1)-\bb_{21}^{\ast}(-1)=0,
\end{align*}
which proves (\ref{spectral21mix}) of \Cref{spectralcond}.

From the definition of $\mathcal{I}_t\B$ we see that
\begin{align*}
{(\mathcal{I}_t\B)^{\ast}_{11}}'(-1)+2(\mathcal{I}_t\B)^{\ast}_{12}(-1)=\: &
-\tfrac{1}{2}(\bb_{11}^{\ast}(-1)+\bb_{21}^{\ast}(-1))\\
& -\tfrac{1}{2}(\bb_{12}^{\ast}(-1)-\bb_{11}^{\ast}(-1)-\bb_{21}^{\ast}(-1)+\bb_{22}^{\ast}(-1))\\
=\: &0.
\end{align*}
Furthermore, by multiplying \cref{int12} with the factor $(z^{-1}-1)$, differentiating this equation with respect to $z$, substituting $z=1$ and using the Taylor conditions, we obtain
\begin{equation*}
(\mathcal{I}_t\B)^{\ast}_{12}(1)=-\tfrac{1}{2}\Big({\bb^{\ast}_{12}}'(1)-{\bb^{\ast}_{11}}'(1)+{\bb^{\ast}_{22}}'(1)-{\bb^{\ast}_{21}}'(1)
\Big).
\end{equation*}
This implies
\begin{equation*}
{(\mathcal{I}_t\B)^{\ast}_{11}}'(1)-2(\mathcal{I}_t\B)^{\ast}_{12}(1)=2\varphi,
\end{equation*}
where $\varphi$ is defined by $\varphi=\tfrac{1}{2}({\bb^{\ast}_{12}}'(1)+{\bb^{\ast}_{22}}'(1)-1)$. This proves property (\ref{spectral11mix}) of \Cref{spectralcond}, concluding the proof of the lemma.
\end{proof}

In \Cref{masklemma1} and \Cref{masklemma2} we defined two operators $\partial_t$ and $\mathcal{I}_t$ which are inverse to each other. Denote by $\ell_s$ be the set of all masks satisfying the spectral condition of \cref{spectral} and 
by $\ell_t$ the set of all masks satisfying the Taylor conditions of \Cref{taylor}. Then
\begin{align}
\partial_t: \quad &\ell_s \to \ell_t \hspace{2cm}  \mathcal{I}_t: \quad \ell_t \to \ell_s
\end{align}
and it easy to verify that 
\begin{equation}
\partial_t(\mathcal{I}_t \B)=\B \quad \text{and} \quad \mathcal{I}_t(\partial_t \A)=\A.
\end{equation}

\subsection{Relations between converging vector and Hermite schemes}\label{subsec:conv}
In the previous section we derived a one-to-one correspondence between a mask satisfying the spectral condition and a mask satisfying the Taylor conditions. For masks of converging schemes we formulate a result based on 
Theorem 21 in \citet{merrien12}, and on the results of \Cref{sec:propertiestaylor}.
\begin{theorem}\label{hermitetaylorconv}
A $C^{\ell}, \ell \geq 0,$ {\vs} $S_{\B}$ satisfying the Taylor conditions with limit functions with vanishing first component, gives rise to an $HC^{\ell+1}$ Hermite scheme $S_{\A}$ satisfying the spectral condition.
\end{theorem}
In the next lemma we show that the condition of vanishing first component in the limits generated by $S_{\B}$ can be replaced by a condition on the mask $\B$. This also follows from results in \citet{micchelli98}.
\renewcommand{\bc}{\boldsymbol{c}}
\begin{lemma}\label{formlimit}
 Let $S_{\B}$ be a convergent {\vs}. Denote by $\Psi_{\bc}={\psi_{1,\bc} \choose \psi_{2,\bc}}$ the limit function generated from the initial data $\bc \in \ell(\mathbb{R}^2)$. Then
 \begin{equation*}
 \mathcal{E}_{\B}=\operatorname{span}\{e_2\} \iff \psi_{1,\bc}=0 \text{ for all initial data } \bc.
 \end{equation*}
\end{lemma}
\begin{proof}
 First we show that $\mathcal{E}_{\B}=\operatorname{span}\{e_2\}$ implies $\psi_{1,\bc}=0$ for all $\bc$. This follows from the observation that $\Psi_{\bc}(x)\in \mathcal{E}_{\B}$ for all $x \in \mathbb{R}$.
 The observation follows from the convergence of $S_{\B}$ to a continuous limit and from the basic refinement rules for large $k$
 \begin{equation*}
  (S_{\B}^{k+1}\bc)_{2i}=\sum_{j \in \m}B_{2j}(S_{\B}^{k}\bc)_{i-j}, \quad (S_{\B}^{k+1}\bc)_{2i+1}=\sum_{j \in \m}B_{2j+1}(S_{\B}^{k}\bc)_{i-j}, \text{ for } i \in \m.
 \end{equation*}
To prove the other direction we use the proof of Theorem 2.2 in \citet{cohen96}. It shows that 
\begin{equation}\label{limM}
 \lim_{n \to \infty}M_{\B}^n=\int_{\mathbb{R}} \Phi(x) dx,
\end{equation}
where $M_{\B}$ is defined in \Cref{eigenvaluesmod}, and $\Phi$ is the limit function generated by $S_{\B}$ from the initial data $\delta I_2$. Here $I_2$ is the identity matrix of dimension 2 and $\delta \in \ell(\mathbb{R})$ satisfies $\delta_0=1$, $\delta_i=0, i \neq 0, i \in \m $,
or equivalently ${\phi_{1j}(x) \choose \phi_{2j}(x) }$ is the limit from the initial data $\delta e_j$ for $j \in \{1,2\}$. Thus
\begin{equation*}
 \phi_{11}(x)=\phi_{12}(x)=0 \quad \text{for } x \in \mathbb{R}.
\end{equation*}
It follows from \cref{limM} that 
\begin{equation}\label{limM2}
 \lim_{n \to \infty}M_{\B}^n=\begin{pmatrix} 
                              0 & 0 \\
                              \nu & \theta
                             \end{pmatrix}, \quad \nu, \theta \in \mathbb{R}.
\end{equation}
Assume $\mathcal{E}_{\B} \neq \operatorname{span}\{e_2\}$. Then $\mathcal{E}_{\B}=\mathbb{R}^2$, and $M_{\B}=I_2$, since by \Cref{eigenvaluesmod} the eigenspace of $M_{\B}$ with respect to $1$ is exactly $\mathcal{E}_{\B}$.
Thus $\lim_{n \to \infty}M_{\B}^n=I_2$ in contradiction to \cref{limM2}.
\end{proof}

\subsection{Imposing the Taylor conditions}
Denote by $\tilde{\ell}_t \subsetneqq \ell_t$ the set of masks satisfying $\B \in \ell_t$ and $\mathcal{E}_{\B}=\operatorname{span}\{e_2\}$.
It follows from \Cref{hermitetaylorconv} and \Cref{formlimit} that for $\B \in \tilde{\ell}_t$, a mask of a $C^{\ell}$ {\vs}, if also $\mathcal{I}_1\B \in \tilde{\ell}_t$, then $\mathcal{I}_1\B$ is a mask of a $C^{\ell+1}$ {\vs} which is the Taylor scheme of a $HC^{\ell+2}$ Hermite scheme.
The next results
show that $\mathcal{I}_1(\tilde{\ell}_t) \subseteq \tilde{\ell}_t$ does not hold in general. Nevertheless, in the following we construct a transformation $\mathcal{R}$ such that 
$\mathcal{R}^{-1}(\mathcal{I}_1\B) \mathcal{R} \in \tilde{\ell}_t$ for $\B \in \tilde{\ell}_t$.

First we look for a canonical transformation of a mask $\B \in \ell_t$ to define $\mathcal{I}_1\B$.
\begin{lemma}\label{eigenvaluesM}
Let $\B \in \ell_t$. Then $M_{\B}$ has the eigenvalue $1$ with eigenvector $\tiny \Spvek{0;1}$ and the eigenvalue $\tfrac{1}{2}\bb^{\ast}_{11}(1)$ with eigenvector $\tiny \Spvek{1;-1}$.
A canonical transformation and its inverse are
\begin{equation*}
\overline{R}=\Big( \begin{array}{r r}0 & 1\\ 1 & -1 \end{array}\Big) \quad \text{with inverse} \quad
\overline{R}^{-1}=\Big( \begin{array}{r r}1 & 1\\ 1 & 0 \end{array}\Big).
\end{equation*}
\end{lemma}
\begin{proof}
From the Taylor conditions we immediately get
\begin{equation*}
M_{\B}=\tfrac{1}{2}(B^0+B^1)=\tfrac{1}{2}\B^{\ast}(1)=\begin{pmatrix}\tfrac{1}{2}\bb_{11}^{\ast}(1) & 0 \\[0.2cm] \tfrac{1}{2}\bb_{21}^{\ast}(1) & 1 \end{pmatrix}.
\end{equation*}
The eigenvalues of $M_{\B}$ can now be read from the diagonal. Also, it is clear that ${0 \choose 1}$ is an eigenvector with eigenvalue $1$. For the other eigenvector we use the Taylor condition (\ref{tmask3}) (in Definition \ref{taylor}) in the third equality below, and obtain
\begin{align*}
M_{\B}\Big(\begin{array}{r}1 \\ -1\end{array}\Big)&=\begin{pmatrix}\tfrac{1}{2}\bb_{11}^{\ast}(1) & 0 \\[0.2cm] \tfrac{1}{2}\bb_{21}^{\ast}(1) & 1 \end{pmatrix}\Big(\begin{array}{r}1 \\ -1\end{array}\Big)=
\begin{pmatrix}\tfrac{1}{2}\bb_{11}^{\ast}(1) \\[0.2cm] \tfrac{1}{2}\bb_{21}^{\ast}(1)-1\end{pmatrix}=
\begin{pmatrix}\tfrac{1}{2}\bb_{11}^{\ast}(1) \\[0.2cm] -\tfrac{1}{2}\bb_{11}^{\ast}(1)\end{pmatrix}\\
&=\tfrac{1}{2}\bb_{11}^{\ast}(1)\Big(\begin{array}{r}1 \\ -1\end{array}\Big). \qedhere
\end{align*}
The structure of $\overline{R}$ follows directly from \cref{R}.
\end{proof}

\Cref{eigenvaluesM} leads to
\begin{theorem}\label{taylortotaylor}
Let $\B \in \tilde{\ell}_t$ and let its associated vector scheme $S_{\B}$ be convergent. Let $\mathcal{I}_1$ be the smoothing operator for {\vs}s in \cref{vec:soperator}.
Then $\mathcal{I}_1\B \in \tilde{\ell}_t$ if and only if the Laurent polynomial $\bb_{11}^{\ast}(z)+\bb_{21}^{\ast}(z)-\bb_{12}^{\ast}(z)-\bb_{22}^{\ast}(z)$ has a root at 1 of multiplicity at least $2$. 
\end{theorem}

\begin{proof}
From \Cref{rem:span} we know that $\mathcal{E}_{\mathcal{I}_1\B}=\mathcal{E}_{\B}=\operatorname{span}\{e_2\}$. Furthermore,
recall from \cref{vec:soperator} that $\mathcal{I}_1\B=\overline{R}(\mathcal{I}_1\overline{\B})\overline{R}^{-1}$ with $\overline{\B}=\overline{R}^{-1}\B \overline{R}.$ In \Cref{eigenvaluesM} a canonical transformation $\overline{R}$ is computed.
Therefore $\overline{\B}$ is given by
\begin{equation}\label{eqproof}
\overline{\B}=\Big(\begin{array}{c c}
\overline{\bb}_{11}& \overline{\bb}_{12}\\
\overline{\bb}_{21} & \overline{\bb}_{22}
\end{array}\Big)=
\Big(\begin{array}{c c}
\bb_{12}+\bb_{22} & \bb_{11}+\bb_{21}-\bb_{12}-\bb_{22}\\
\bb_{12} & \bb_{11}-\bb_{12}
\end{array}\Big).
\end{equation}
The parts of the Taylor conditions concerning the elements of $\B^{\ast}(1)$ imply that the symbol $\overline{\bb}^{\ast}_{12}(z)$ has a root at 1.
Therefore there exists a Laurent polynomial $\boldsymbol{\kappa}^{\ast}(z)$ such that $\overline{\bb}^{\ast}_{12}(z)=(z^{-1}-1)\boldsymbol{\kappa}^{\ast}(z)$.
Combining \cref{eqproof} with \cref{defI} we obtain
\begin{equation*}
(\mathcal{I}_1\overline{\B})^{\ast}(1)=\Big(\begin{array}{c c}
2 & \tfrac{1}{2}\boldsymbol{\kappa}^{\ast}(1)\\[0.2cm]
0 & \tfrac{1}{2}\bb^{\ast}_{11}(1)
\end{array}\Big) \quad \text{and} \quad
(\mathcal{I}_1\overline{\B})^{\ast}(-1)=\Big(\begin{array}{c c}
0 & \tfrac{1}{2}\boldsymbol{\kappa}^{\ast}(-1)\\[0.2cm]
0 & \tfrac{1}{2}\bb^{\ast}_{11}(-1)
\end{array}\Big).
\end{equation*}
Therefore
\begin{align}\label{IB}
(\mathcal{I}_1\B)^{\ast}(1)&=\overline{R}(\mathcal{I}_1\overline{\B})^{\ast}(1)\overline{R}^{-1}=\Big(\begin{array}{c c}
\tfrac{1}{2}\bb_{11}^{\ast}(1) & 0\\[0.1cm]
2+\tfrac{1}{2}(\boldsymbol{\kappa}^{\ast}(1)-\bb^{\ast}_{11}(1)) & 2
\end{array}\Big) \quad \text{and} \\ \nonumber
(\mathcal{I}_1\B)^{\ast}(-1)&=\overline{R}(\mathcal{I}_1\overline{\B})^{\ast}(-1)\overline{R}^{-1}=\Big(\begin{array}{c c}
\tfrac{1}{2}\bb^{\ast}_{11}(-1) & 0\\[0.2cm]
\tfrac{1}{2}(\boldsymbol{\kappa}^{\ast}(-1)-\bb^{\ast}_{11}(-1)) & 0
\end{array}\Big).
\end{align}
By \cref{IB}, (\ref{tmask1}) and (\ref{tmask2}) of the Taylor conditions in \Cref{taylor} are satisfied by $\mathcal{I}_1{\B}$.
The mask $\mathcal{I}_1{\B}$ satisfies (\ref{tmask3}) of the Taylor conditions if and only if $\boldsymbol{\kappa}^{\ast}(1)=0$. By the definition of $\boldsymbol{\kappa}$,
this is equivalent to the Laurent polynomial $\overline{\bb}^{\ast}_{12}(z)=\bb_{11}^{\ast}(z)+\bb_{21}^{\ast}(z)-\bb_{12}^{\ast}(z)-\bb_{22}^{\ast}(z)$
having a root of multiplicity $2$ at $1$.
\end{proof}

Thus, in general, $\mathcal{I}_1(\tilde{\ell}_t)\nsubseteq \tilde{\ell}_t$. In the next two lemmas we solve this problem.
\begin{lemma}\label{transformtotaylor}
Let $\B$ be a mask of a converging {\vs} satisfying $\mathcal{E}_{\B}=\operatorname{span}\{e_2\}$ and $\bb_{11}^{\ast}(1)\neq 2$. Then there exists a transformation $\mathcal{R}$ such that $\widetilde{\B}=\mathcal{R}^{-1}\B \mathcal{R} \in \tilde{\ell}_t$.
\end{lemma}
\begin{proof}
First we note that by \Cref{rem:e2}, the mask ${\B}$ satisfies (\ref{tmask1}) and (\ref{tmask2}) of the Taylor conditions and obtain
\begin{equation*}
\B^{\ast}(1)=\begin{pmatrix}
			a & 0\\
			b & 2
			\end{pmatrix},
\end{equation*}
with $a,b \in \mathbb{R}$ and $a\neq 2$ by the assumption of the lemma. To impose (\ref{tmask3}) of the Taylor conditions we take $\mathcal{R}$
with a second column $e_2$ in order to retain the above second columns.
A normalized choice of the first column of $\mathcal{R}$ yields
\begin{equation}\label{mathcalR}
\mathcal{R}=\Big(\begin{array}{r r}
			1 & 0\\
			\eta & 1
			\end{array}\Big),
\quad 
\mathcal{R}^{-1}=\Big(\begin{array}{r r}
			1 & 0\\
			-\eta & 1
			\end{array}\Big),
\end{equation}
and we obtain
\begin{equation*}
\widetilde{\B}^{\ast}(1)=\begin{pmatrix}
			a & 0\\
			(2-a)\eta +b & 2
			\end{pmatrix}.
\end{equation*}
To satisfy (\ref{tmask3}) of the Taylor conditions $(2-a)\eta +b+a=2$. Therefore we choose $\eta=1+\frac{b}{a-2}$. 
From the form of $\widetilde{\B}^{\ast}(1)$ and since $a\neq 2$, we see that $\mathcal{E}_{\widetilde{\B}}=\operatorname{span}\{e_2\}$.
\end{proof}

Next we show that we can apply the smoothing procedure and transform the resulting mask to a mask in $\tilde{\ell}_t$.
\begin{corollary}\label{tilde}
Let $\B \in \tilde{\ell}_t$ such that $S_{\B}$ is a $C^{\ell}$ {\vs}, for $\ell \geq 0$. Then $\widetilde{\mathcal{I}_1(\B)}\in \tilde{\ell}_t$ and $S_{\widetilde{\mathcal{I}_1(\B)}}$ is a $C^{\ell+1}$ {\vs}.
\end{corollary}
\begin{proof}
It follows from \Cref{rem:span} that $\mathcal{E}_{\mathcal{I}_1\B}=\mathcal{E}_{\B}=\operatorname{span}\{e_2\}$.
\Cref{IB} implies $(\mathcal{I}_1\B)^{\ast}_{11}(1)=\tfrac{1}{2}\bb_{11}^{\ast}(1)$. From \Cref{eigenvaluesM} we know that $\tfrac{1}{2}\bb_{11}^{\ast}(1)$ is an eigenvalue of $M_{\B}$. 
By \Cref{alggeomulti}, $\frac{1}{2}|\bb_{11}^{\ast}(1)|\leq 1$. In particular $(\mathcal{I}_1\B)^{\ast}_{11}(1)\neq 2$. Therefore, $\mathcal{I}_1\B$ satisfies the conditions of \Cref{transformtotaylor} and with the transformation
$\mathcal{R}$ in \cref{mathcalR}, $\mathcal{R}^{-1}(\mathcal{I}_1\B) \mathcal{R} \in \tilde{\ell}_t$. The statement about smoothness follows from the construction of $\mathcal{I}_1$ in \cref{vec:soperator}.
\end{proof}
\subsection{A procedure for increasing the smoothness of Hermite schemes}
\Cref{taylortotaylor} and \Cref{tilde} allow to define the following procedure for increasing the smoothness of {\hs}s:
\begin{alg}\label{alghermite}
The input is a mask $\A$ satisfying the spectral condition (\Cref{spectralcond}). Furthermore we assume that its Taylor scheme is $C^{\ell-1}$ for $\ell \geq 1$ and that the limit functions have vanishing first component for all input data (this implies that $S_{\A}$ is $HC^{\ell}$).
The output is a mask $\mathbf{C}$ which satisfies the spectral condition and its associated Hermite scheme $S_{\mathbf{C}}$ is $HC^{\ell+1}$.
\begin{enumerate}
\item Compute the Taylor scheme $\partial_t \A$ (\Cref{masklemma1}).
\item Apply \Cref{algorithm} and \Cref{transformtotaylor} to obtain  $\B=\widetilde{\mathcal{I}_1(\partial_t\A)}$.
\item Define $\mathbf{C}=\mathcal{I}_t(\mathbf{B})$ (\Cref{masklemma2}). 
\end{enumerate}
\end{alg}

In the following we execute \Cref{alghermite} for a general mask $\A$ satisfying the assumptions of the procedure, and present explicitly $\C^{\ast}(z)$.

\renewcommand{\bc}{\boldsymbol{\gamma}}

From the definition of $\eta$ in the proof of \Cref{transformtotaylor} it is easy to see that $\eta=\frac{\ab^{\ast}_{12}(1)}{2-\ab_{22}^{\ast}(1)}$. 
This is well-defined, since $M_{\partial_t\A}$ has $\ab^{\ast}_{22}(1)$ as an eigenvalue. By \Cref{alggeomulti}, $\ab^{\ast}_{22}(1)\neq 2$.
Then with $\zeta=\eta+1$ we get
\begin{align} \label{c1}
\bc^{\ast}_{11}(z) = & \ \tfrac{1}{2}(z^{-1}+1)\Big(
\ab^{\ast}_{12}(z)\Big(({\zeta}-{\zeta}^2)z^{-3}+{\zeta}^2z^{-2}+({\zeta}^2-1)z^{-1}-({\zeta}^2+{\zeta})\Big)\\ \nonumber
&+\ab_{11}^{\ast}(z)\Big({\zeta}(z^{-1}-1)(1-{\zeta})+{\zeta}\Big)
+\ab_{22}^{\ast}(z)({\zeta}(z^{-2}-1)-1)({\zeta}-1)\\ \nonumber
&+\ab_{21}^{\ast}(z)({\zeta}^2-{\zeta})
\Big),\\ \nonumber
\bc^{\ast}_{12}(z) = & \ \tfrac{1}{2}\Big(
\ab^{\ast}_{12}(z)\Big((1-{\zeta})^2z^{-3}+{\zeta}(1-{\zeta})z^{-2}+{\zeta}(1-{\zeta})z^{-1}+{\zeta}^2\Big)\\ \nonumber
&+ \ab^{\ast}_{22}(z)\Big(-(z^{-2}-1)(1-{\zeta})^2+{\zeta}-1 \Big)\\ \nonumber
&+\ab^{\ast}_{11}(z)\Big((z^{-1}-1)(1-{\zeta})^2+1-{\zeta} \Big)
 -\ab^{\ast}_{21}(z)(1-{\zeta})^2
\Big)\Big/ (z^{-1}-1),\\ \nonumber
\bc^{\ast}_{21}(z)= & \ \tfrac{1}{2}(z^{-2}-1)\Big(
\ab^{\ast}_{12}(z)\Big(-{\zeta}^2z^{-3}+({\zeta}+{\zeta}^2)(z^{-2}+z^{-1})-({\zeta}+1)^2 \Big)\\ \nonumber
&+\ab^{\ast}_{11}(z){\zeta}(1-{\zeta}(z^{-1}-1))+\ab^{\ast}_{22}(z){\zeta}({\zeta}(z^{-2}-1)-1)+{\zeta}^2\ab^{\ast}_{21}(z)
\Big),\\ \nonumber
\bc^{\ast}_{22}(z) = & \ \tfrac{1}{2}\Big(
\ab^{\ast}_{12}(z)\Big(({\zeta}^2-{\zeta})z^{-3}+(1-{\zeta}^2)z^{-2}-{\zeta}^2z^{-1}+({\zeta}^2+{\zeta}) \Big)\\ \nonumber
& +\ab^{\ast}_{11}(z)(1-{\zeta})(1-{\zeta}(z^{-1}-1))+\ab^{\ast}_{22}(z){\zeta}((1-{\zeta})(z^{-2}-1)+1)\\ \nonumber
&+\ab^{\ast}_{21}(z)({\zeta}-{\zeta}^2)
\Big).
\end{align}
In the special case $\ab^{\ast}_{12}(1)=0$, $\zeta=1$, $\C^{\ast}(z)$ reduces to
\begin{align}\label{c2}
\bc^{\ast}_{11}(z) = & \ \tfrac{1}{2}(z^{-1}+1)\Big((z^{-2}-2)\ab^{\ast}_{12}(z)+\ab_{11}^{\ast}(z)\Big),\\ \nonumber
\bc^{\ast}_{12}(z) = & \ \tfrac{1}{2}\frac{\ab^{\ast}_{12}(z)}{(z^{-1}-1)},\\ \nonumber
\bc^{\ast}_{21}(z)= & \ \tfrac{1}{2}(z^{-2}-1)\Big(\ab_{21}^{\ast}(z)-\ab_{11}^{\ast}(z)(z^{-1}-2) \\ \nonumber
&+\ab_{22}^{\ast}(z)(z^{-2}-2)-\ab_{12}^{\ast}(z)(z^{-1}-2)(z^{-2}-2)\Big),\\ \nonumber
\bc^{\ast}_{22}(z) = & \ \tfrac{1}{2}(\ab^{\ast}_{22}(z)-(z^{-1}-2)\ab_{12}^{\ast}(z)).
\end{align}
With the explicit form of $\C$, we can prove
\begin{lemma}\label{smoothspectral}
 Let $\varphi_{\A}$ be the constant corresponding to the spectral condition in \cref{spectral} satisfied by $\A$. Then the constant corresponding to the spectral condition satisfied by $\C$ is $\varphi_{\C}=\varphi_{\A}-\frac{1}{2}$.
\end{lemma}
In particular, the application of \Cref{alghermite} to interpolatory {\hs}s does not result in interpolatory {\hs}s.

\begin{proof}
Differentiating $\bc_{11}^{\ast}(z)$ and $\bc_{12}^{\ast}(z)$ given in \cref{c1}, and evaluating at $z=1$ we obtain in view of condition (\ref{spectral11mix}) in \Cref{spectralcond}
\begin{align*}
2 \varphi_{\C}={\bc_{11}^{\ast}}'(1)-2\bc_{12}^{\ast}(1)=& \: {\ab_{11}^{\ast}}'(1)-2\ab_{12}^{\ast}(1)
+({\zeta}-1)({\ab_{21}^{\ast}}'(1)-2\ab_{22}^{\ast}(1))\\
&+2({\zeta}-1)+\frac{1}{2}\ab_{12}^{\ast}(1)-{\zeta}-\frac{1}{2}\ab_{22}^{\ast}(1)(1-{\zeta})\\
=&\: 2\varphi_{\A} +\frac{1}{2}(\ab_{12}^{\ast}(1)-\ab_{22}^{\ast}(1))-\frac{1}{2}{\zeta}(2-\ab_{22}^{\ast}(1))\\
=&\: 2(\varphi_{\A} -\frac{1}{2}).\qedhere
\end{align*} 
\end{proof}

From the explicit form of $\C$ we can infer
\begin{corollary}\label{supportlength}
Let $\A$ and $\C$ be masks as in \Cref{alghermite}. If $\A$ has support contained in $[-N_1,N_2]$ with $N_1,N_2 \in \mathbb{N}$, then the support of $\C$ is contained in $[-N_1-5,N_2]$.
\end{corollary}
Therefore \Cref{alghermite} increases the support length at most by 5.
\begin{corollary}\label{a12}
Let $\A$ be a mask satisfying the spectral condition of \cref{spectral} and let its associated Taylor scheme be convergent.
Assume that $\ab_{12}^{\ast}(1)=0$ (i.e.\ $\zeta=1$). Denote by $\C$ the mask obtained via \Cref{alghermite}. Then $\bc_{12}^{\ast}(1)=0$ if and only if ${\ab_{12}^{\ast}}'(1)=0$.
\end{corollary}
\begin{proof}
From the definition of $\C$ in \cref{c1} it is easy to see that $\bc_{12}^{\ast}(1)=-\tfrac{1}{2}{\ab_{12}^{\ast}}'(1)$. Therefore $\bc_{12}^{\ast}(1)=0$ iff ${\ab_{12}^{\ast}}'(1)=0$.
\end{proof}

Let $r$ be the multiplicity of the root at 1 of $\ab_{12}^{\ast}(z)$. \Cref{a12} implies that $r-1$ iterations of the smoothing procedure stay within the special case of $\zeta=1$.

\begin{example}\label{example}
We consider the Hermite scheme generating $C^1$ piecewise cubic polynomials interpolating the inital data (see \cite{merrien92}). The mask of the scheme is given by
\begin{equation*}
A_{-1}=\left(\begin{array}{r r} \frac{1}{2} & -\frac{1}{8} \\[6pt] \frac{3}{4} & -\frac{1}{8} \end{array}\right), \quad A_{0}=\left(\begin{array}{r r} 1 & 0\\[6pt] 0 & \frac{1}{2} \end{array}\right), \quad
A_{1}=\left(\begin{array}{r r} \frac{1}{2} & \frac{1}{8}\\[6pt] -\frac{3}{4} & -\frac{1}{8} \end{array}\right).
\end{equation*}

It is easy to see that it satisfies the spectral condition of \cref{spectral} with $\varphi_{\A}=0$. 
In \citet{merrien12} it is proved that its Taylor scheme is convergent with limit functions of vanishing first component (and thus the original {\hs} is $HC^1$).

We apply \Cref{alghermite} to this scheme to obtain a new {\hs} of regularity $HC^2$, using the explicit expressions in \cref{c1} and \cref{c2}.
First we compute the symbol:
\begin{equation*}
\A^{\ast}(z)=\left(\begin{array}{c c} \frac{1}{2}(1+z)^2z^{-1} & -\frac{1}{8}(1-z^2)z^{-1} \\[6pt] \frac{3}{4}(1-z^2)z^{-1} & -\frac{1}{8}z^{-1}+\frac{1}{2}- \frac{1}{8}z \end{array}\right).
\end{equation*}
Note that $\ab^{\ast}_{12}(1)=0$ with multiplicity $1$. Therefore we are in the special case $\zeta=1$.

We apply \cref{c2} and obtain the symbol of $\C$:
\begin{equation*}
 \C^{\ast}(z)=\frac{1}{16}
 \left(\begin{array}{c c} (z^{-1}+1)^2(-z^{-2}+z^{-1}+6+2z) & -z-1 \\[6pt] 
 (z^{-2}-1)\Big(z^{-4}-3z^{-3}-3z^{-2}+13z^{-1}+6\Big) &
 z^{-2}-3z^{-1}+3+z \end{array}\right).
\end{equation*}
From \Cref{smoothspectral} we also know that $\C$ satisfies the spectral condition with $\varphi_{\C}=-\tfrac{1}{2}$.
Therefore the {\hs} associated with $\C$ is an $HC^2$ scheme which is not interpolatory. A basic limit function of this scheme is depicted in \Cref{fig:example1}.
Note that the support of $\C$ is $[-6,1]$ and has thus increased from length of $3$ to the length of $8$.

If we want to apply another round of \Cref{alghermite}, we have to use \cref{c1} with $\zeta=\tfrac{14}{15}$.
\end{example}
\begin{figure}
\centering

	\begin{minipage}[b]{.45\textwidth}
	\centering

	\begin{overpic}[scale=0.2]{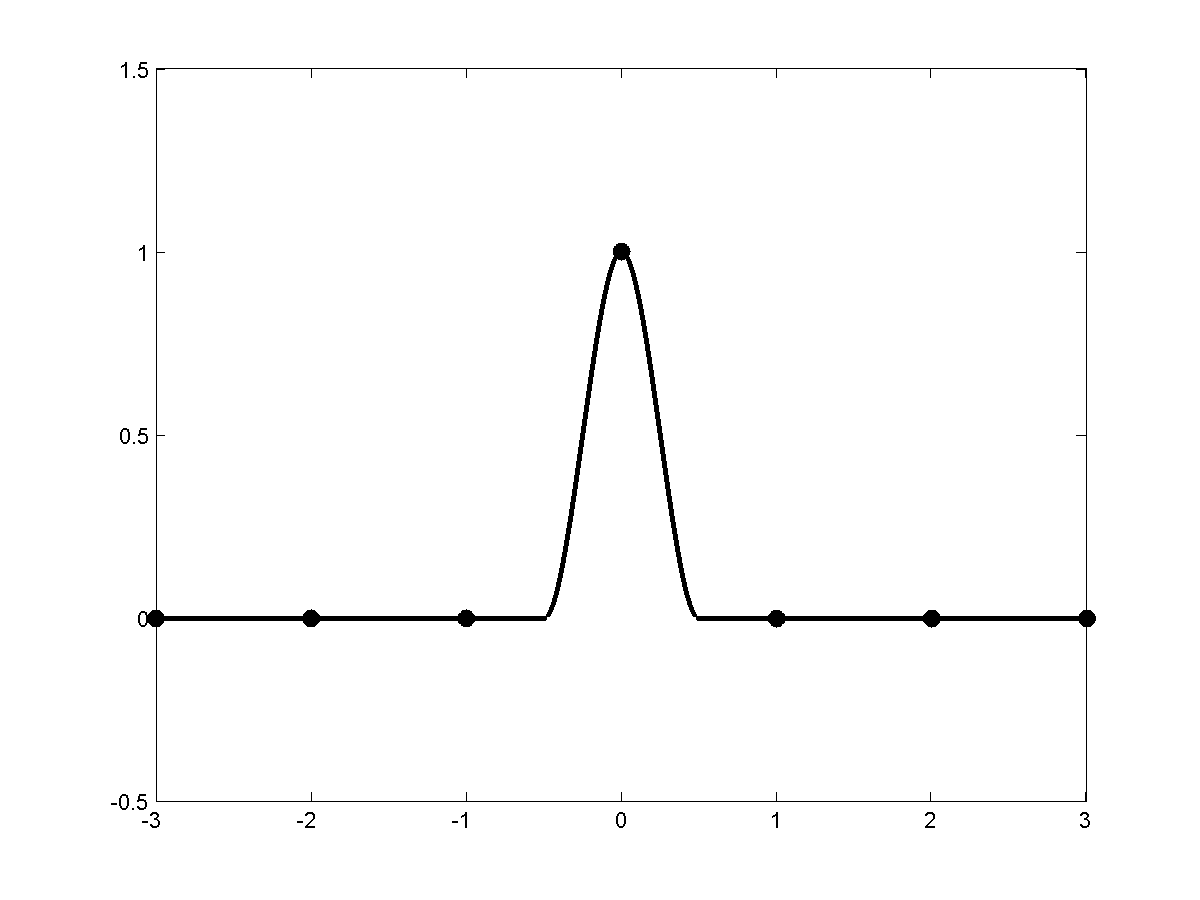}
	 \put(-10,35){$f(t)$}
	 \put(50,-4){$t$} 
	\end{overpic}
	
	\end{minipage}
	\hspace{-1.5cm}
	\begin{minipage}[b]{.45\textwidth}
	\centering

	\begin{overpic}[scale=0.2]{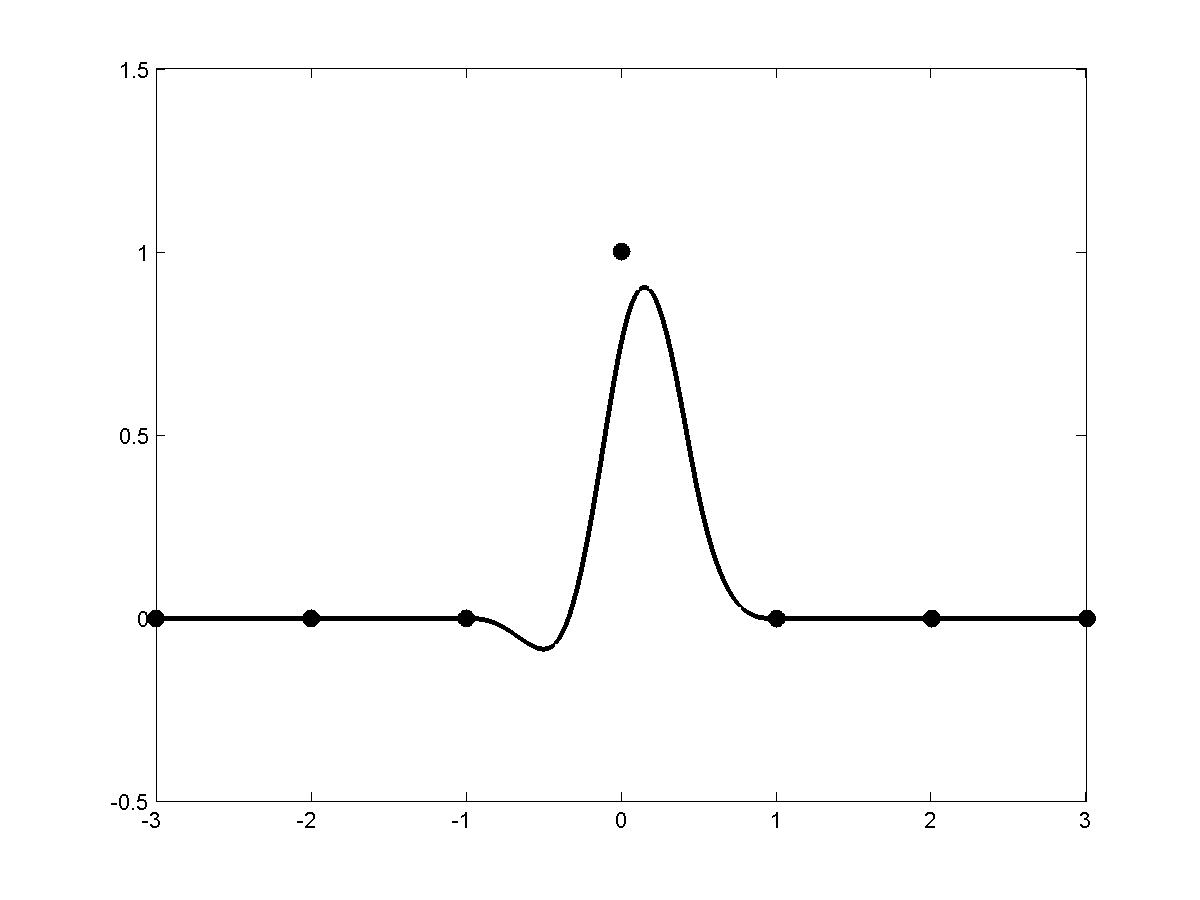}
	 \put(-10,35){$g(t)$}
	 \put(50,-4){$t$} 
	\end{overpic}
	
	\end{minipage}\\[2.5ex]
	
	\begin{minipage}[b]{.45\textwidth}
	\centering

	\begin{overpic}[scale=0.2]{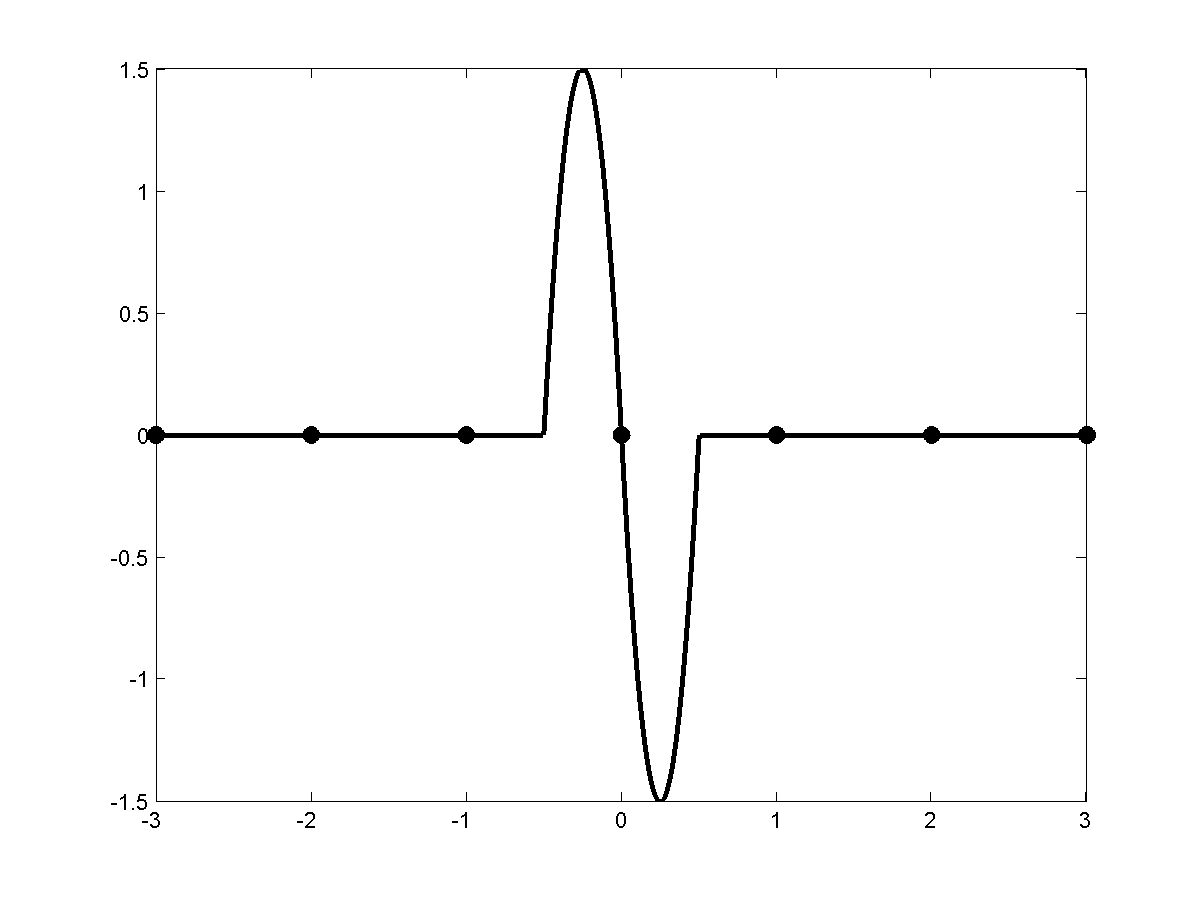}
	 \put(-11,35){$f'(t)$}
	 \put(50,-4){$t$} 
	\end{overpic}
	
	\end{minipage}
	\hspace{-1.5cm}
	\begin{minipage}[b]{.45\textwidth}
	\centering

	\begin{overpic}[scale=0.2]{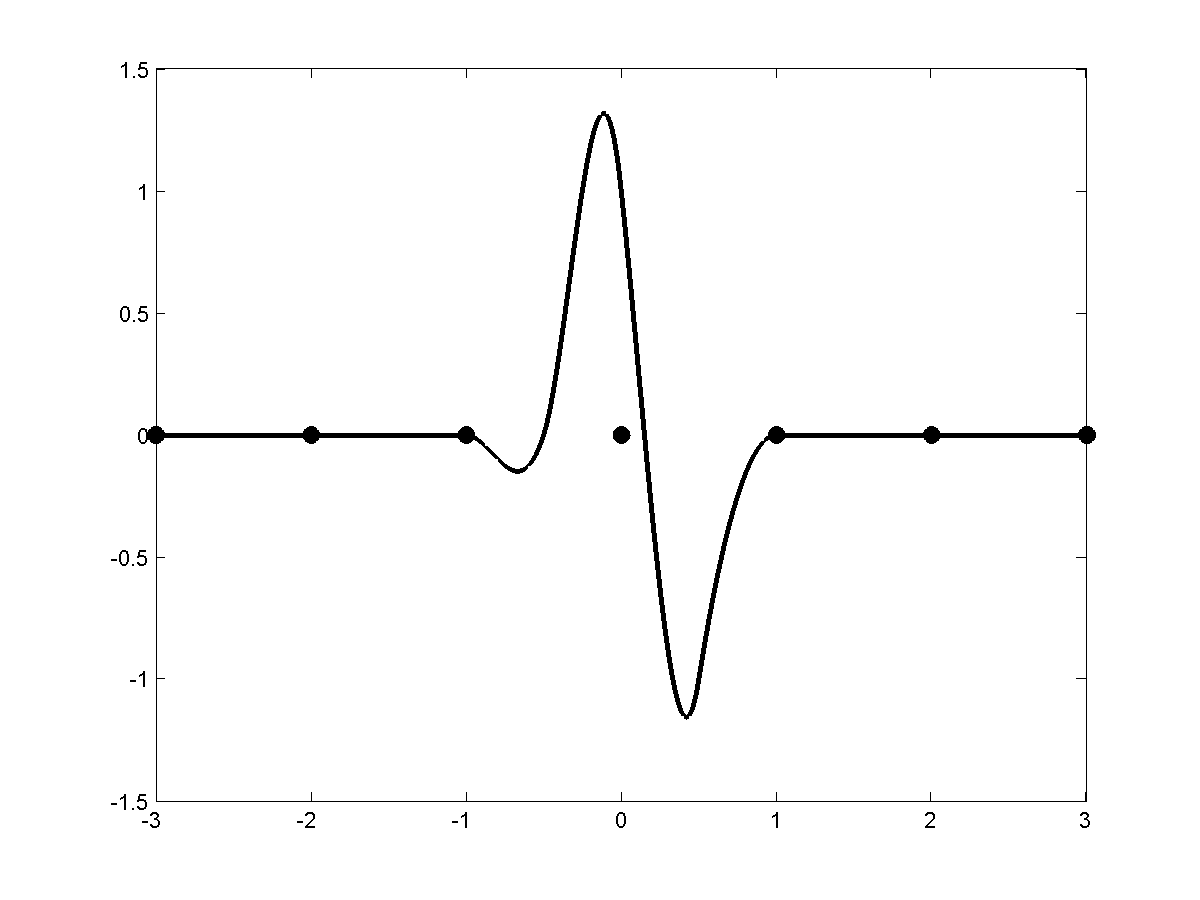}
	 \put(-11,35){$g'(t)$}
	 \put(50,-4){$t$} 
	\end{overpic}
	
	\end{minipage}
	
	\caption{Basic limit functions and their first derivatives of the {\hs}s of \Cref{example}. \emph{First column:} interpolatory $HC^1$ scheme $S_{\A}$ with basic limit function $f$. 
	\emph{Second column:} the smoothed non-interpolatory $HC^2$ scheme $S_{\C}$ with basic limit function $g$.}
	\label{fig:example1}
\end{figure}

\begin{figure*}[!b]
\centering

	\begin{minipage}[b]{.45\textwidth}
	\centering

	\begin{overpic}[scale=0.2]{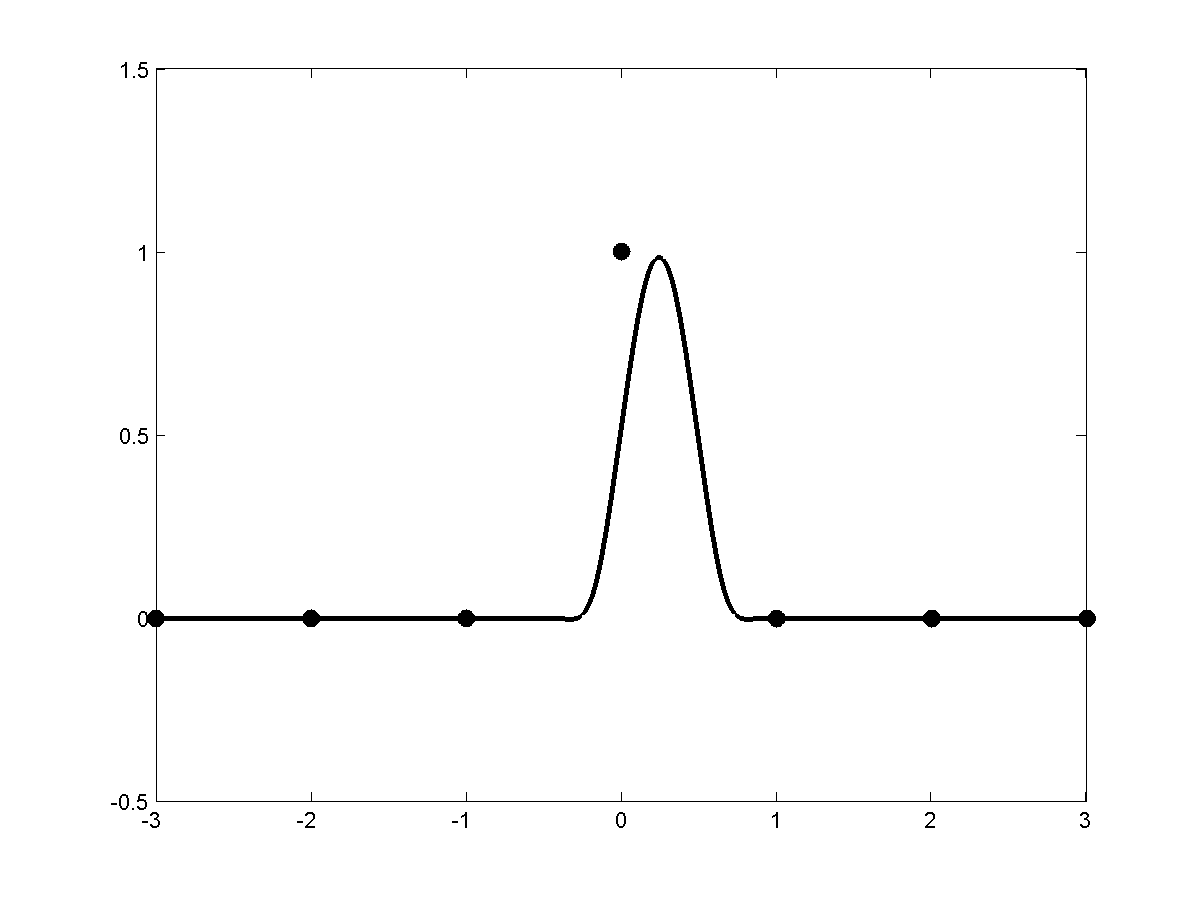}
	 \put(-10,35){$f(t)$}
	 \put(50,-4){$t$} 
	\end{overpic}
	
	\end{minipage}
	\hspace{-1.5cm}
	\begin{minipage}[b]{.45\textwidth}
	\centering

	\begin{overpic}[scale=0.2]{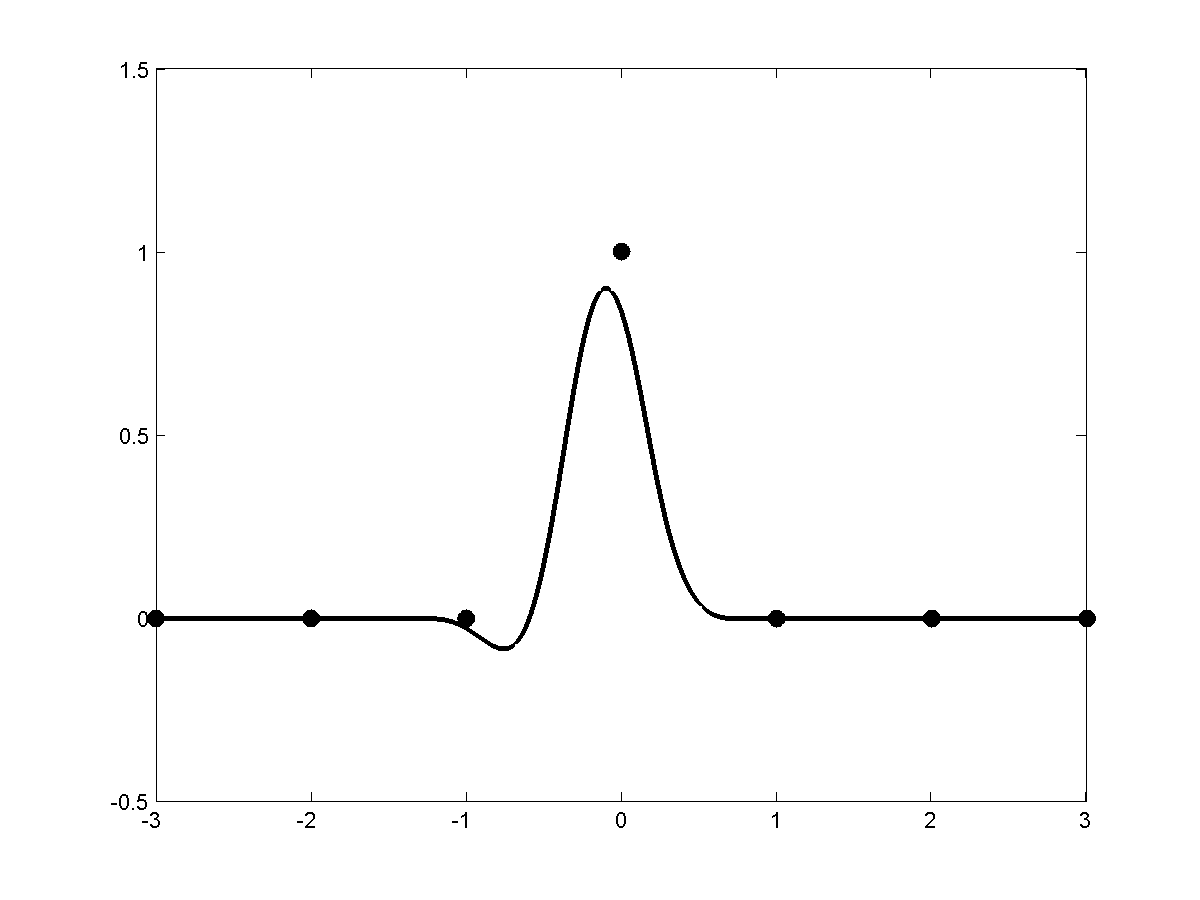}
	 \put(-10,35){$g(t)$}
	 \put(50,-4){$t$} 
	\end{overpic}
	
	\end{minipage}\\[2.5ex]
	
	\begin{minipage}[b]{.45\textwidth}
	\centering

	\begin{overpic}[scale=0.2]{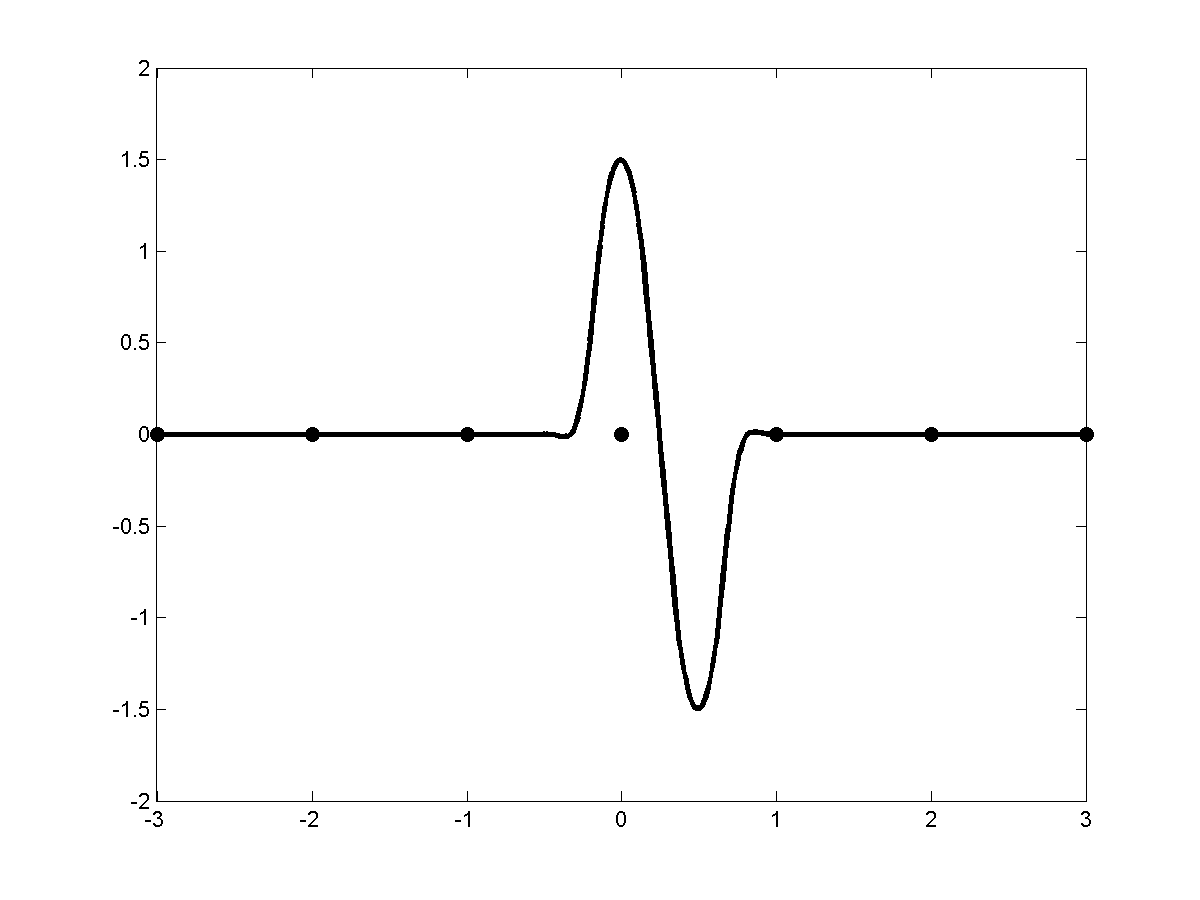}
	 \put(-11,35){$f'(t)$}
	 \put(50,-4){$t$} 
	\end{overpic}
	
	\end{minipage}
	\hspace{-1.5cm}
	\begin{minipage}[b]{.45\textwidth}
	\centering

	\begin{overpic}[scale=0.2]{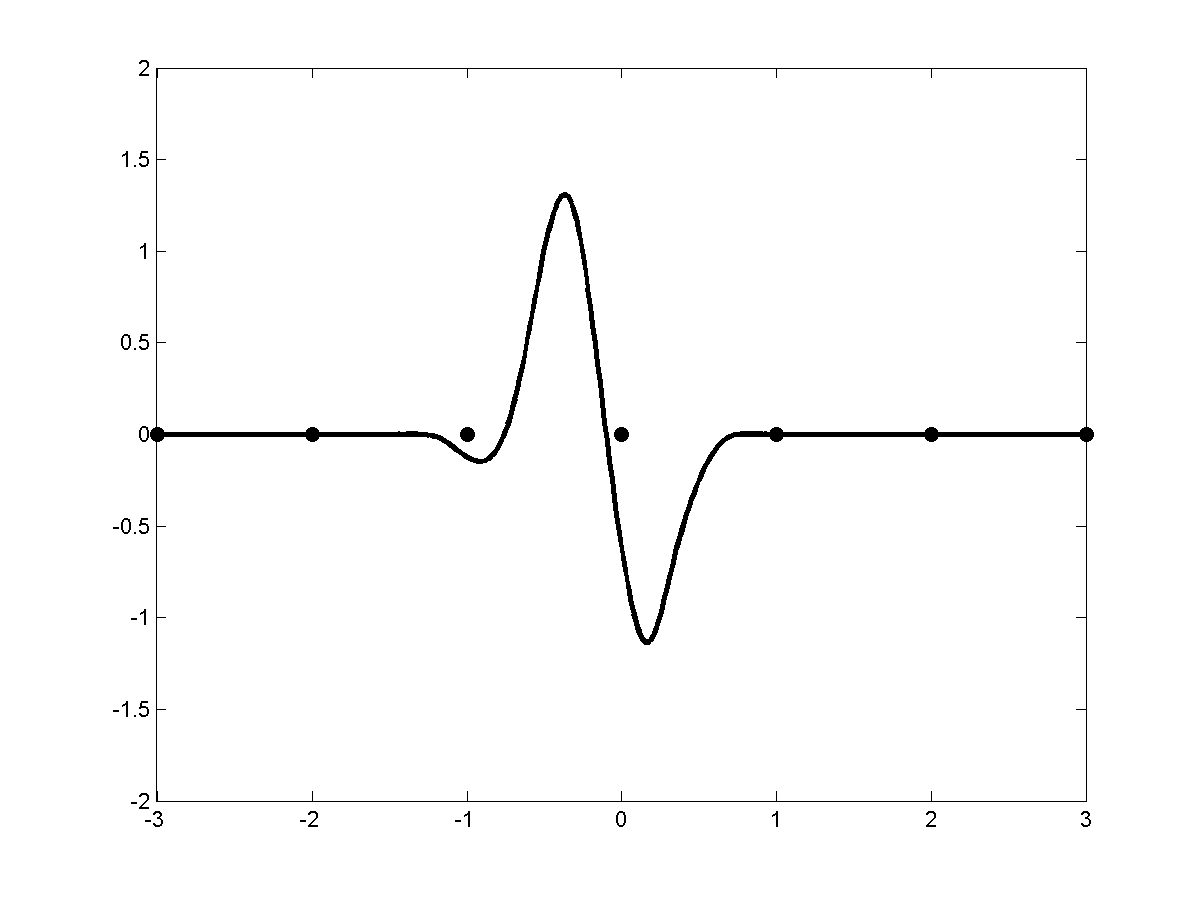}
	 \put(-11,35){$g'(t)$}
	 \put(50,-4){$t$} 
	\end{overpic}
	
	\end{minipage}\\[2.5ex]
	
	\begin{minipage}[b]{.45\textwidth}
	\centering

	\begin{overpic}[scale=0.2]{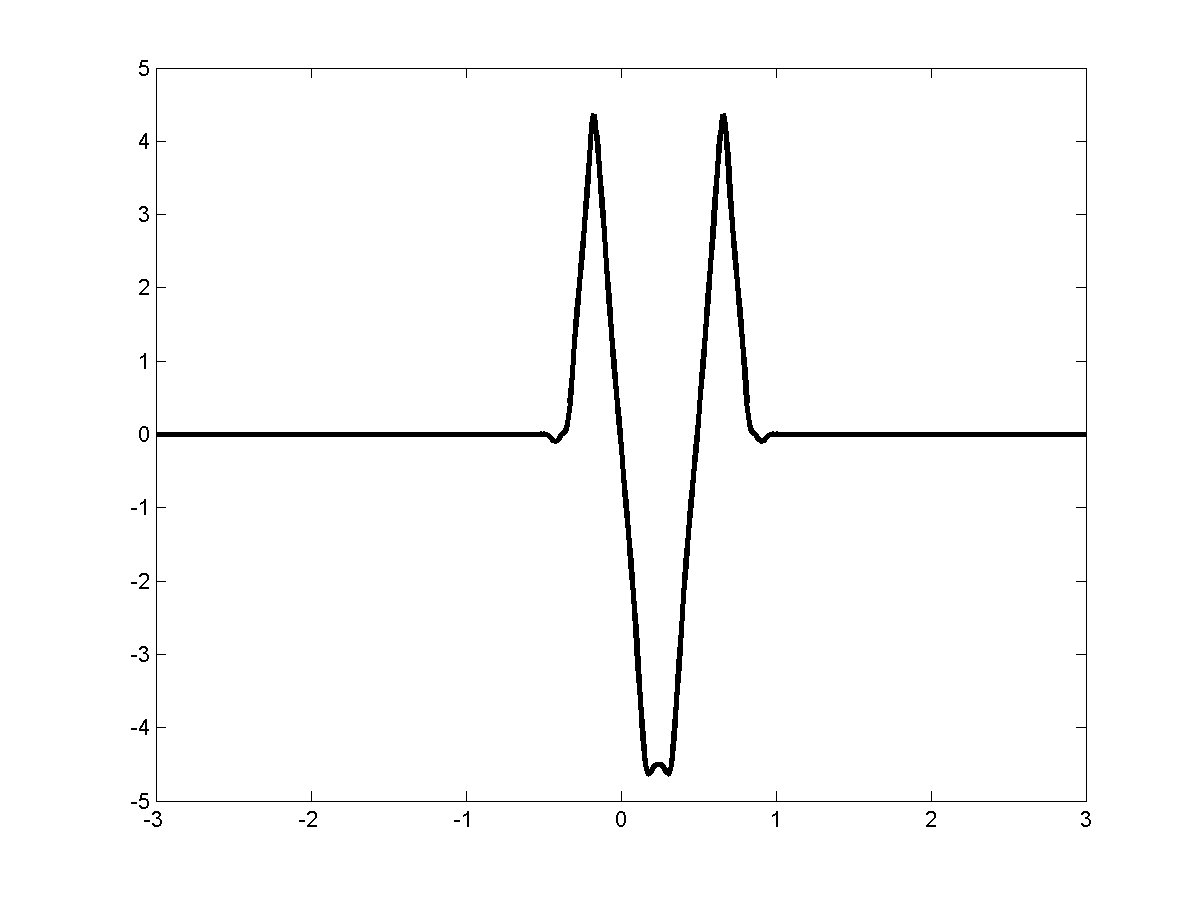}
	 \put(-12,35){$f''(t)$}
	 \put(50,-4){$t$} 
	\end{overpic}
	
	\end{minipage}
	\hspace{-1.5cm}
	\begin{minipage}[b]{.45\textwidth}
	\centering

	\begin{overpic}[scale=0.2]{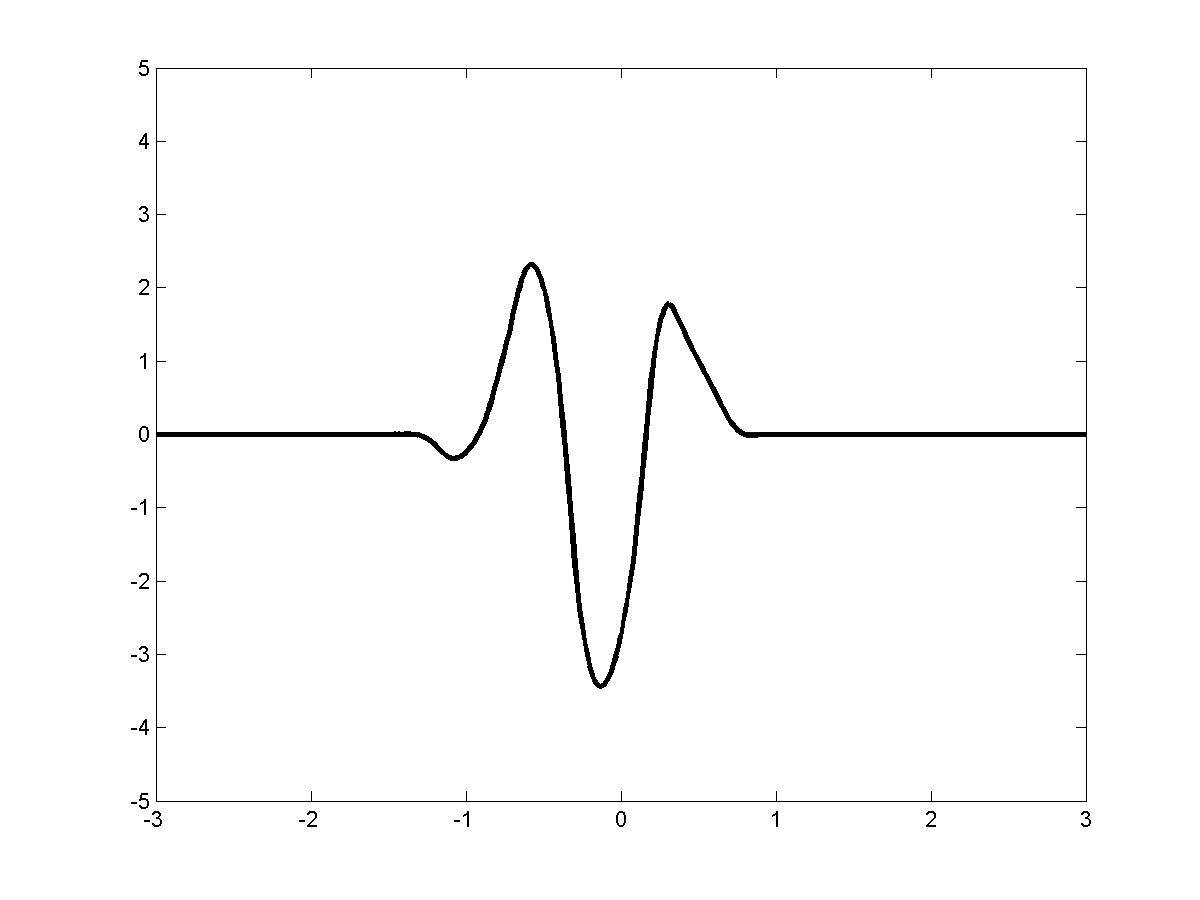}
	 \put(-12,35){$g''(t)$}
	 \put(50,-4){$t$} 
	\end{overpic}
	
	\end{minipage}
	
	\caption{Basic limit functions, their first and second derivatives of the {\hs}s of \Cref{derham}. \emph{First column:} non-interpolatory $HC^2$ scheme $S_{\A}$ with basic limit function $f$.
	\emph{Second column:} smoothed non-interpolatory $HC^3$ scheme $S_{\C}$ with basic limit function $g$.}
	\label{fig:example2}
\end{figure*}
\begin{example}\label{derham}
We consider one of the de Rham-type {\hs}s of \citet{dubuc08} obtained from the scheme of \Cref{example}. Its mask is given by 
\begin{align*}
&A_{-2}=\frac{1}{8}\left(\begin{array}{r r} \frac{5}{4} & -\frac{3}{8} \\[6pt] \frac{9}{2} & -\frac{5}{4} \end{array}\right), \quad
A_{-1}=\frac{1}{8}\left(\begin{array}{r r} \frac{27}{4} & -\frac{9}{8} \\[6pt] \frac{9}{2} & \frac{3}{4} \end{array}\right),\\[0.3cm]
&A_{0}=\frac{1}{8}\left(\begin{array}{r r} \frac{27}{4} & \frac{9}{8} \\[6pt] -\frac{9}{2} & \frac{3}{4} \end{array}\right),\quad
A_{1}=\frac{1}{8}\left(\begin{array}{r r} \frac{5}{4} & \frac{3}{8} \\[6pt] -\frac{9}{2} & -\frac{5}{4} \end{array}\right).
\end{align*}
It is easy to see that it satisfies the spectral condition of \cref{spectral} with $\varphi_{\A}=-\tfrac{1}{2}$. 
In \citet{conti14} it is proved that its Taylor scheme is $C^1$ with limit functions of vanishing first component (and thus the original {\hs} is $HC^2$).

We apply Algorithm \ref{alghermite} to this scheme to obtain a new {\hs} of regularity $HC^3$.
First we compute the symbol:
\begin{equation*}
\A^{\ast}(z)=\frac{1}{16}\left(\begin{array}{c c} \frac{1}{2}(z^{-1}+1)(5z+2z+5z^{-1}) & -\frac{3}{4}(z^{-1}-1)(z+4+z^{-1}) \\[6pt] 9(z^{-2}-1)(z+1) & \frac{1}{2}(z^{-1}+1)(-5z+8-5z^{-1}) \end{array}\right).
\end{equation*}
Note that $\ab^{\ast}_{12}(1)=0$ with multiplicity $1$. Therefore, as in \Cref{example}, we are in the special case $\zeta=1$.
We apply \cref{c2} and obtain the symbol of $\C$:
\begin{align*}
\bc^{\ast}_{11}(z)&=\frac{1}{128}(z^{-1}+1)(-3z^{-4}-9z^{-3}+25z^{-2}+75z^{-1}+36+4z),\\
\bc^{\ast}_{12}(z)&=-\frac{3}{128}(z+4+z^{-1}),\\
\bc^{\ast}_{21}(z)&=\frac{1}{128}(z^{-2}-1)\Big(3z^{-5}-7z^{-4}-37z^{-3}+37z^{-2}+128z^{-1}+20-8z\Big),\\
\bc^{\ast}_{22}(z)&=\frac{1}{128}(3z^{-3}-7z^{-2}-21z^{-1}+21-4z).
\end{align*}
We also know from \Cref{smoothspectral} that $\C$ satisfies the spectral condition with $\varphi_{\C}=-1$. Therefore the {\hs} associated with $\C$ is an $HC^3$ scheme which is not interpolatory.
A basic limit function of this scheme is depicted in \Cref{fig:example2}. Note that the support of $\C$ is $[-7,1]$ and has thus increased from length of $4$ to the length of $9$.

If we want to apply another round of \Cref{alghermite}, we have to use \cref{c1} with $\zeta=\tfrac{41}{44}$.
\end{example}

\section*{Acknowledgements}
Most of this research was done while the first author was with TU Graz. It was supported by the Austrian Science fund (grant numbers W1230 and I705).
The authors thank Costanza Conti, Tomas Sauer and Johannes Wallner for their valuable comments and suggestions.

\bibliographystyle{plainnat}


\end{document}